\documentclass[reqno]{amsart}
\usepackage{hyperref}
\usepackage{color}

\def\N{{\mathbb{N}}}

\def\R{{\mathbb{R}}}

\begin{document}
\title[Infinite-dimensional multiobjective optimal control]{Infinite-dimensional multiobjective optimal control in continuous time}
\author[N. Hayek, H. Yilmaz]{ Naila Hayek \& Hasan Yilmaz}

\address{Naila Hayek: Laboratoire CRED EA 7321,\newline
Universit\'{e} Panth\'{e}on-Assas Paris 2, Paris, France.}
\email{naila.hayek@u-paris2.fr}
\address{Hasan Yilmaz:  Laboratoire SAMM EA 4543,\newline
Universit\'{e} Paris 1 Panth\'{e}on-Sorbonne, centre P.M.F.,\newline
90 rue de Tolbiac, 75634 Paris cedex 13,
France.}
\email{yilmaz.research@gmail.com}
\date{June, 29, 2022}
\numberwithin{equation}{section}
\newtheorem{theorem}{Theorem}[section]
\newtheorem{lemma}[theorem]{Lemma}
\newtheorem{example}[theorem]{Example}
\newtheorem{remark}[theorem]{Remark}
\newtheorem{definition}[theorem]{Definition}
\newtheorem{corollary}[theorem]{Corollary}
\newtheorem{proposition}[theorem]{Proposition}
\thispagestyle{empty} \setcounter{page}{1}
\maketitle
\vskip3mm
\noindent
{\bf Mathematical  Subject Classification 2010}:
\vskip4mm
%%%%%%%%%%%%%%%%%%%%%%%%%%%%%%%%%%%%%%%%%%%%%%%%%%%%%%%%%%%%%%%%%%%%%%%%%%%%%%%%
%
\begin{abstract}  

This paper studies multiobjective optimal control problems  in the continuous-time framework when the
space of states and the space of controls are infinite-dimensional and with lighter smoothness assumptions than the usual ones. The paper generalizes to the multiobjective case existing results for single-objective optimal control problems in that framework. The dynamics are governed by differential equations and a finite number of terminal equality and inequality constraints are present.  Necessary conditions of Pareto optimality are provided namely Pontryagin maximum principles in the strong form.  Sufficient conditions are also provided. 
%Other notions of Pareto optimality are defined when the infinite series do not necessarily converge and links with these unbounded cases are established. 
%\textbf{200} words.
\end{abstract}
\section{Introduction}
In this paper we study multiobjective optimal control problems, with open loop information structure, in the continuous-time framework, when the space of states and the space of controls are infinite-dimensional. We  derive necessary conditions and sufficient
conditions of Pareto optimality. We rely on lighter smoothness assumptions than the usual ones. The paper extends to the multiobjective case, results obtained for single-objective optimal control problems  in infinite dimension.

%Multiobjective optimal control was first studied by Zadeh \cite{Zad}  who was followed by Salukvadze \cite{Sal}, Yu and Leitmann \cite{Yu}, Ishizuka and Shimizu \cite{Ish}, Khanh and Nuong \cite{Khan}.

 In the continuous-time framework, some results of
% infinite-horizon
multiobjective optimal control problems can be found in  Bellaassali and Jourani \cite{BJ}, in Zhu \cite{Zh}, in Bonnel and Kaya \cite{BonKaya}, in Gramatovici \cite{G}, in de Oliveira and Nunes Silva \cite{Oliv}  and in references therein. \\
Differential games are widely used in economic theory, see \cite{L}, \cite{DJ}, \cite{RE} ,\cite{E} and \cite{S} and Pareto optimality plays a central role in analyzing these
problems. 
In the discrete-time framework, results on infinite-horizon multiobjective optimal control problems  can be found in Hayek \cite{H1} and \cite{H2}, \cite{H3}, in Ngo-Hayek \cite{NgoH}.
Bachir and Blot \cite{BB1}, \cite{BB2} extended infinite-horizon single-objective optimal control problems in the discrete-time framework, to the case of infinite-dimensional spaces of states and controls and Hayek \cite{H4} extended these results to multiobjective optimal control problems.

   In this paper we rely on the results of Blot and Yilmaz in \cite{BY} and \cite{BY1} to study multiobjective optimal control problems in an infinite-dimensional setting and in continuous time. We obtain necessary conditions of Pareto optimality under the form of Pontryagin Principles and we provide sufficient conditions of Pareto optimality.
   
We start by providing necessary conditions of optimality for Mayer  multiobjective optimal control problems and we deduce necessary conditions for Bolza problems with lighter smoothness assumptions. 
The Hadamard differential of a mapping between Banach spaces,
which is stronger than the G\^ateaux differential but weaker than the Fr\'echet differential, has been applied many times in the literature.  In finite dimension, the Hadamard differential  coincides with the Fr\'echet  differential , but for infinite-dimensional spaces the Fr\'echet differential  is much stronger, even for Lipschitz functions.
%On the other hand, if f is Lipschitz, then the Hadamard derivative coincideswith the Göateaux derivative

We provide different results relying on different constraint qualifications  namely  to obtain non trivial multipliers associated to the objective functions. 
For the sufficient conditions we follow Mangasarian \cite{M} and Seierstadt-Sydsaeter \cite{SS} and we rely on weaker assumptions than the usual ones namely the concavity at a point and the quasi-concavity at a point.

The plan of this paper is as follows. Section 2 is devoted to definitions and assumptions. In section 3 the problems are presented:  multiobjective optimal control problems governed by a differential equation when the
space of states and the space of controls are infinite-dimensional, in the continuous-time framework. The notions of Pareto optimality and weak Pareto optimality are defined.  In section 4 the theorems on necessary conditions of Pareto optimality are stated namely Pontryagin maximum principles in the strong form for a Mayer's problem and for a Bolza's problem.  In section 5 we give sufficient conditions. The proofs of the necessary conditions theorems are provided in section 6 and those of the sufficient ones in section 7.

\section{Definitions and assumptions}
We set $\N$ the set of positive integers and $\N^*=\N\setminus\{0\}$. $\R$ denotes the set of real numbers and $\R_+$ the set of non-negative real numbers.\\
When $X$ and $Y$ are Hausdorff space, $C^0(X,Y)$ denotes the space of continuous mappings from $X$ into $Y$.\\
When $Y$ be a Hausdorff space and $T\in \R_+^*=]0,+\infty[$. As in \cite{BY}, a function $u : [0,T] \rightarrow Y$ is called piecewise continuous when there exists a subdivision $0= \tau_0 < \tau_1 < ...< \tau_k < \tau_{k+1} = T$ such that
\begin{itemize}
\item For all $i \in \{0,...,k \}$, $u$ is continuous on $]\tau_i, \tau_{i+1}[$.
\item For all $i \in \{0,...,k \}$, the right-hand limit $u(\tau_i +)$ exists in $Y$.
\item For all $i \in \{1,...,k+1 \}$, the left-hand limit $u(\tau_i -)$ exists in $Y$.
\end{itemize}
The space of piecewise continuous mappings from $[0,T]$ to $Y$ is denoted by \\$PC^0([0,T], Y)$.\\
A function $u \in PC^0([0,T], Y)$ is called a normalized piecewise continuous function when moreover $u$ is right continuous on $[0,T[$ and when $u(T -) = u(T)$ cf. \cite{BY}.\\
We denote by $NPC^0([0,T],Y)$ the space of such functions.\\
As in \cite{BY}, when $Y$ is a real Banach space, a function $x : [0,T] \rightarrow Y$ is called piecewise continuously differentiable when $x \in C^0([0,T],Y )$ and there exists a subdivision $(\tau_i)_{0 \leq i \leq k+1}$ of $[0,T]$ such that the following conditions are fulfilled.
\begin {itemize}
\item For all $i \in \{0,...,k \}$, $x$ is continuously differentiable on $]\tau_i, \tau_{i+1}[$
\item For all $i \in \{0,...,k \}$, $x'(\tau_i +)$ exists in $Y$
\item For all $i \in \{1,...,k+1 \}$, $x'(\tau_i -)$ exists in $Y$
\end{itemize}
The $(\tau_i)_{1\le i\le k+1}$ are the corners of the function $x$.\\
We denote by $PC^1([0,T], Y)$ the space of such functions.\\
When $G$ is an open subset of $Y$, $PC^1([0,T], G)$ is the set of functions \\$x \in PC^1([0,T], Y)$ such that $x([0,T]) \subset G$.\\
When $x \in PC^1([0,T], Y)$ and $(\tau_i)_{0 \leq i \leq k+1}$ are the corners of the function $x$, we define the function $\underline{d}x : [0,T] \rightarrow Y$, called the extended derivative of $x$, by setting
\begin{equation}\label{eq21}
\underline{d}x(t) :=
\left\{
\begin{array}{ccl}
x'(t) & {\rm if} & t \in [0,T] \setminus \{ \tau_i : i \in \{0, ..., k+1 \}\}\\
x'(\tau_i +) & {\rm if} & t = \tau_i, i \in \{0,...,k \}\\
x'(T-) & {\rm if} & t = T.
\end{array}
\right.
\end{equation}
Notice that, contrary to the usual derivative of $x$, the extended derivative of $x$ is defined on $[0,T]$ all over.
Note that $\underline{d}x \in NPC^0([0,T], Y)$ and we have the following relation between $x$, $\underline{d}x$ and the Riemann integral:  
$$\text{for all } a<t \text{ in } [0,T],\, x(t)-x(a)=\int_{a}^{t} \underline{d}x(s)ds,$$
Besides, $\underline{d}$ is a bounded linear operator from $PC^1([0,T],Y)$ into $NPC^0([0,T],Y)$.\\
All these properties motivated the authors of \cite{BY} to introduce the notion of extended derivative for piecewise continuously differentiable functions. \\
When $X$ and $Y$ are real normed vector spaces, ${\mathcal L}(X,Y)$ denotes the space of the bounded linear mappings from $X$ into $Y$ and $X^*$ denotes the topological dual of $X$.\\
We denote by $\|\cdot\|_\mathcal{L}$ the usual norm of ${\mathcal L}(X,Y)$.\\ 
Let $G$ be a non-empty open subset of $X$, let $\mathfrak{f}: G \rightarrow Y$ be a mapping and let $x\in G$.\\
The mapping $\mathfrak{f}$ is called G\^ateaux differentiable at $x$ when there exists $D_G\mathfrak{f}(x)\in {\mathcal L}(X,Y)$ such that for all $h\in X$, $\lim_{t \rightarrow 0+} \frac{\mathfrak{f}(x+th)-\mathfrak{f}(x)}{t}=D_G\mathfrak{f}(x)\cdot h$.\\
Moreover, $D_G\mathfrak{f}(x)$ is called the G\^ateaux differential of $\mathfrak{f}$ at $x$. \\
We say that $\mathfrak{f}$ is Hadamard differentiable at $x$ when there exists $D_H\mathfrak{f}(x)\in {\mathcal L}(X,Y)$ such that for each $K$ compact in $X$, $\lim_{t \rightarrow 0+} \sup_{h \in K}\|\frac{\mathfrak{f}(x+th)-\mathfrak{f}(x)}{t}-D_H\mathfrak{f}(x)\cdot h\|=0$.\\ 
Moreover, $D_H\mathfrak{f}(x)$ is called the Hadamard differential of $\mathfrak{f}$ at $x$.\\
When $\mathfrak{f}$ is Hadamard differentiable at $x$, $\mathfrak{f}$ is also G\^ateaux differentiable at $x$ and $D_H\mathfrak{f}(x)=D_G\mathfrak{f}(x)$. But the converse is false in general when the dimension of $X$ is greater than 2. \\
Notice that Hadamard differentiability and G\^ateaux differentiability always coincide for locally Lipschitz functions in any normed vector space.
When it exists, $D_{F}\mathfrak{f}(x)$ denotes the Fr\'echet differential of $\mathfrak{f}$ at $x$.\\ 
When $\mathfrak{f}$ is Fr\'echet differentiable at $x$, $\mathfrak{f}$ is Hadamard differentiable at $x$ and $D_F\mathfrak{f}(x)=D_H\mathfrak{f}(x)$. But the converse is false in general when the dimension of $X$ is infinite. \\
When $X$ is a finite product of $n$ real normed spaces, $X = \prod_{1 \leq i \leq n} X_i$, if $k \in \{ 1,...,n \}$, $D_{F,k}\mathfrak{f}(x)$ (respectively $D_{H,k} \mathfrak{f}(x)$, respectively $D_{G,k} \mathfrak{f}(x)$) denotes the partial Fr\'echet (respectively Hadamard, respectively G\^ateaux) differential of $\mathfrak{f}$ at $x$ with respect to the $k$-th vector variable. \\
More information on these notions of differentials can be found in \cite{F}.\\
Next, we introduce definitions of notions of concavity at a point in infinite dimension cf. Mangasarian \cite{M} for the finite dimension. This concepts will be used for sufficient conditions.  \\
Let $\mathfrak{g}: G \rightarrow \R$ be a mapping. The mapping ${\mathfrak g}$ is said to be concave at $x$ when for all $y\in G$, for all $t\in[0,1]$ s.t. $(1-t)x+ty\in G$, ${\mathfrak g}((1-t)x+ty) \ge (1-t)\mathfrak{g}(x)+t\mathfrak{g}(y)$.\\ 
When ${\mathfrak g}$ is G\^ateaux differentiable at $x$, the function ${\mathfrak g}$ is said to be pseudo-concave at $x$ when for all $y\in G$, [$D_G{\mathfrak g}(x)\cdot (y-x) \le 0 \Rightarrow {\mathfrak g}(y) \le {\mathfrak g}(x)$].\\
The mapping ${\mathfrak g}$ is said to be quasi-concave at $x$ when for all $y\in G$, for all $t\in[0,1]$ s.t. $(1-t)x+ty\in G$,[$\mathfrak{g}(x) \le \mathfrak{g}(y) \Rightarrow {\mathfrak g}(x) \le {\mathfrak g}((1-t)x+ty)$].\\
When ${\mathfrak g}$ is G\^ateaux differentiable at $x$ and ${\mathfrak g}$ is quasi-concave at $x$, we have, for all $y\in G$, [${\mathfrak g}(y) \ge {\mathfrak g}(x) \Rightarrow D_G{\mathfrak g}(x)\cdot (y-x) \ge 0 $].   
\section{The multiobjective optimal control problems  }
Let $T\in ]0,+\infty[$, $E$ is a real Banach space, $\Omega$ is a non-empty subset of $E$, $U$ is a Hausdorff topological space and $\xi_0\in \Omega$.
We consider the functions $f:[0,T] \times \Omega \times U \rightarrow E$, $f_i^0:[0,T] \times \Omega \times U \rightarrow \R$ when $i\in \{1,...,l\}$, $g_i^0 :\Omega \rightarrow \R$ when $i\in\{1,...,l\}$, $g^\alpha: \Omega \rightarrow \R$ when $\alpha\in\{1,...,m\}$ and $h^\beta : \Omega \rightarrow \R$ when $\beta \in \{1,...,q\}$, when $(l,m,q)\in (\N^*)^3$.
For all $i\in\{1,...,l\}$ we consider also the function $J_i : PC^1([0,T],\Omega)\times NPC^0([0,T],U)\rightarrow \R$ defined by, for all $(x,u)\in PC^1([0,T],\Omega)\times NPC^0([0,T],U)$, $J_i(x,u):=g_i^0(x(T))+\int_{0}^{T} f_i^0(t,x(t),u(t))dt$.\\
With these elements, we can build the following multiobjective Bolza problem  
\[
({\mathcal B})
\left\{
\begin{array}{cl}
{\rm Maximize} & (J_1(x,u),...,J_l(x,u)) \\
{\rm subject \;  to} & x \in PC^1([0,T], \Omega), u \in NPC^0([0,T], U)\\
\null & \forall t\in [0,T],\, \underline{d}x(t) = f(t,x(t), u(t)), \; x(0) = \xi_0\\
\null & \forall \alpha \in\{ 1,...,m\}, \; \; g^{\alpha}(x(T)) \geq 0\\
\null & \forall \beta \in \{1,..., q\}, \; \; h^{\beta}(x(T)) = 0.
\end{array}\right.
\]
Our problem is a reformulation of the multiobjective classical Bolza problem where the controlled dynamical system is formulated as follows : $x'(t)=f(t,x(t),u(t))$ when $x'(t)$ exists, and the control function $u\in PC^0([0,T],U)$. In \cite{BY}, we explain that the present formulation is equivalent  to the classical one,  for the single-objective Bolza problem. By using the same reasoning, we remark that this formulation is also equivalent for the multiobjective Bolza problem.\\
When for all $i\in\{1,...,l\},\, f_i^0=0$, (${\mathcal B}$) is called a multiobjective Mayer problem and it is denoted by (${\mathcal M}$).\\  
We denote by $Adm({\mathcal B})$ (respectively $Adm({\mathcal M})$) the set of the admissible processes of $({\mathcal B})$ (respectively $({\mathcal M})$).\\
It is clear that $Adm({\mathcal B})=Adm({\mathcal M})$.
When $(x,u)$ is an admissible process for $(\mathcal{B})$ or $(\mathcal{M})$, we consider the following constraint qualifications, when the functions defining  the terminal constraints and the terminal parts of the criterion are Hadamard differentiable at $x(T)$.
\[
(QC_0)
\left\{
\begin{array}{l}
{\rm If} \;\; (b_i)_{1\le i\le l}\in \R_+^{l}, (c_{\alpha})_{1 \leq \alpha \leq m} \in \R_+^{ m}, (d_{\beta})_{1 \leq \beta \leq q} \in \R^q {\rm \;\;\;satisfy} \\
(\forall \alpha \in \{ 1,...,m\}, \; c_{\alpha} g^{\alpha}(x(T)) = 0), {\rm and} \\
\sum_{i=1}^{l} b_iD_{H}g_i^0(x(T))+\sum_{\alpha = 1}^m c_{\alpha} D_{H}g^{\alpha}(x(T)) + \sum_{\beta = 1}^q d_{\beta} D_{H}h^{\beta}(x(T)) = 0,\\ {\rm then}
\;\; (\forall i\in\{1,...,l\},\;  b_i=0), \;\;
(\forall \alpha \in \{ 1, ...,m\}, \;  c_{\alpha} = 0) \;\; {\rm and} \;\; \\
 (\forall \beta \in \{ 1, ..., q\}, \; d_{\beta} = 0).
\end{array}
\right.
\]
and
\[
(QC_1)
\left\{
\begin{array}{l}
{\rm If} \;\; (c_{\alpha})_{1 \leq \alpha \leq m} \in \R_+^{ m}, (d_{\beta})_{1 \leq \beta \leq q} \in \R^q {\rm \;\;\;satisfy} \\
(\forall \alpha \in \{ 1,...,m\}, \; c_{\alpha} g^{\alpha}(x(T)) = 0), {\rm and} \\
\sum_{\alpha = 1}^m c_{\alpha} D_{H}g^{\alpha}(x(T)) + \sum_{\beta = 1}^q d_{\beta} D_{H}h^{\beta}(x(T)) = 0, {\rm then}\\
(\forall \alpha \in \{ 1, ...,m\}, \;  c_{\alpha} = 0) \;\; {\rm and} \;\; (\forall \beta \in \{ 1, ..., q\}, \; d_{\beta} = 0).
\end{array}
\right.
\]
\begin{definition}
An admissible process $(\overline{x},\overline{u}) $ for $({\mathcal B})$ is a Pareto optimal solution for $({\mathcal B})$ when there does not exist an admissible process $(x,u)$ for $({\mathcal B})$ such that for all $i\in \{1,...,l\}$, $J_i(x,u) \ge J_i(\overline{x},\overline{u})$ and for some $i_0\in\{1,...,l\}$, $J_{i_0}(x,u) > J_{i_0}(\overline{x},\overline{u})$.
\end{definition} 
\begin{definition}
An admissible process $(\overline{x},\overline{u})$ for $({\mathcal B})$ is a weak Pareto optimal solution for $({\mathcal B})$ when there does not exist an admissible process $(x,u) $ for $({\mathcal B})$  such that for all $i\in \{1,...,l\}$, $J_i(x,u) > J_i(\overline{x},\overline{u})$.
\end{definition}   
Now, we formulate a list of conditions which will become the assumptions of our theorems. Let $(x_0,u_0)$ be an admissible process for $(\mathcal{B})$ or $(\mathcal{M})$.
\vskip1mm
\noindent
{\bf Conditions on the vector field.}
\begin{itemize}
%A3
\item[\bf{(A{\sc v}1)}] $f\in C^0([0,T] \times \Omega \times U,E)$, for all $(t, \xi, \zeta) \in [0,T] \times \Omega \times U$, $D_{G,2}f(t,\xi, \zeta)$ exists, for all $(t,\zeta) \in [0,T]\times U$, $D_{F,2}f(t,x_0(t), \zeta)$ exists and $[(t,\zeta)\mapsto D_{F,2}f(t,x_0(t),\zeta)]\in C^0([0,T]\times U,\mathcal{L}(E,E))$. 
% A8 et (A{\sc int}3)
\item[\bf{(A{\sc v}2)}] For all non-empty compact $K\subset \Omega$, for all non-empty compact $M\subset U$, $\sup_{(t,\xi,\zeta)\in [0,T]\times K\times M} \|D_{G,2}f(t,\xi,\zeta)\|_{\mathcal{L}} <+\infty$.
\end{itemize}
\vskip1mm
\noindent
{\bf Conditions on the integrands of the criterion.}
\begin{itemize}
%A3
\item[\bf{(A{\sc i}1)}] For all $i\in\{1,...,l\}$, $f_i^0\in C^0([0,T] \times \Omega \times U,\R)$, for all $(t, \xi, \zeta) \in [0,T] \times \Omega \times U$, $D_{G,2}f_i^0(t,\xi, \zeta)$ exists, for all $(t,\zeta) \in [0,T]\times U$, $D_{F,2}f_i^0(t,x_0(t), \zeta)$ exists and $[(t,\zeta)\mapsto D_{F,2}f_i^0(t,x_0(t),\zeta)]\in C^0([0,T]\times U,E^*)$. 
% A8 et (A{\sc int}3)
\item[\bf{(A{\sc i}2)}] For all $i\in\{1,...,l\}$, for all non-empty compact $K\subset \Omega$, for all non-empty compact $M\subset U$, $\sup_{(t,\xi,\zeta)\in [0,T]\times K\times M} \|D_{G,2}f_i^0(t,\xi,\zeta)\|_{\mathcal{L}} <+\infty$.
\end{itemize}
\vskip1mm
\noindent
{\bf Conditions on the functions defining  the terminal constraints and terminal parts of the criterion}
\begin{itemize}
% A1
\item[\bf{(A{\sc t}1)}] For all $i \in \{1,...,l\}$, $g_i^0$ is Hadamard differentiable at $x_0(T)$.
% A2
\item[\bf{(A{\sc t}2)}] For all $\alpha \in \{1,...,m\}$, $g^{\alpha}$ is Hadamard differentiable at $x_0(T)$.
% A2
\item[\bf{(A{\sc t}3)}] For all $\beta \in \{1,...,q \}$, $h^{\beta}$ is continuous on a neighborhood of $x_0(T)$ and Hadamard differentiable at $x_0(T)$.
\end{itemize}

%% theorem 2.1

\section{Necessary conditions of Pareto optimality}
\subsection{Necessary conditions of Pareto optimality for the Mayer problem}

\begin{definition}
The Hamiltonian of (${\mathcal M}$) is the function $H_M :[0,T]\times \Omega \times U\times E^* \rightarrow \R$ defined by, for all $(t,x,u,p)\in [0,T]\times \Omega \times U\times E^*$, $H_M(t,x,u,p) := p \cdot f(t,x,u)$.
\end{definition}
%% theorem 2.1
\begin{theorem}\label{th22} (Pontryagin Principle for the Mayer problem)\\
When $(x_0,u_0)$ is a Pareto optimal solution of $(\mathcal{M})$, under (A{\sc v}1), (A{\sc v}2), (A{\sc t}1), (A{\sc t}2) and (A{\sc t}3),  there exists $(\theta_i)_{1 \le i \le l}\in \R^l$, $(\lambda_{\alpha})_{1 \leq \alpha \leq m} \in \R^{m}$, $(\mu_{\beta})_{1 \leq \beta \leq q} \in \R^q$ and an adjoint function $p \in PC^1([0,T], E^*)$ which satisfy the following conditions.
\begin{itemize}
\item[(NN)] $((\theta_i)_{1 \le i\le l},(\lambda_{\alpha})_{1 \leq \alpha \leq m}$, $(\mu_{\beta})_{1 \leq \beta \leq q})\neq 0$
\item[(Si)] For all $i\in \{1,...,l\}$, $\theta_i\ge 0$ and for all $\alpha \in \{1,...,m \}$, $\lambda_{\alpha} \geq 0$.
\item[(S${\ell}$)] For all $\alpha \in \{1,...,m \}$, $\lambda_{\alpha} g^{\alpha}(x_0(T)) = 0$.
\item[(TC)] $\sum_{i=1}^{l} \theta_iD_Hg_i^0(x_0(T))+\sum_{\alpha = 1}^{m} \lambda_{\alpha} D_Hg^{\alpha}(x_0(T)) + \sum_{\beta = 1}^q \mu_{\beta} D_Hh^{\beta}(x_0(T)) = p(T)$.
\item[(AE.M)] $\underline{d}p(t) = - D_{F,2}H_M(t, x_0(t),u_0(t), p(t))$ for all $t \in [0,T]$.
\item[(MP.M)] For all $t \in [0,T]$, for all $\zeta \in U$, \\
$H_M(t, x_0(t), u_0(t), p(t)) \geq H_M(t, x_0(t), \zeta, p(t))$.
\item[(CH.M)] $\bar{H}_M := [ t \mapsto H_M(t, x_0(t), u_0(t), p(t))] \in C^0([0,T], \R)$.
\end{itemize} 
(NN) is a condition of non nullity, (Si) is a sign condition, (S${\ell}$) is a slackness condition, (TC) is the transversality condition, (AE.M) is the adjoint equation, (MP.M) is the maximum principle and (CH.M) is a condition of continuity on the Hamiltonian.
\end{theorem}
\begin{corollary}\label{cor21}
In this setting and under the assumptions of Theorem \ref{th22}, if moreover we assume that ($QC_1$) is fulfilled for $(x,u) = (x_0,u_0)$, then, for all $t \in [0,T]$, $((\theta_i)_{1\le i \le l}, p(t))$ is never equal to zero. 
\end{corollary}
\begin{corollary}\label{cor22}
In this setting and under the assumptions of Theorem \ref{th22}, if moreover we assume that ($QC_0$) is fulfilled for $(x,u) = (x_0,u_0)$, then, for all $t \in [0,T]$, $p(t)$ is never equal to zero. 
\end{corollary}
As in \cite{BY1}, we introduce another condition
\begin{itemize}
\item[\bf{(A{\sc v}3)}] $U$ is a subset of a real normed vector space $Y$, there exists $\hat{t}\in [0,T]$ s.t. $U$ is a neighborhood of $u_0(\hat{t})$ in $Y$, $D_{G,3}f(\hat{t},x_0(\hat{t}),u_0(\hat{t}))$ exists and it is surjective.
\end{itemize}
\begin{corollary}\label{cor23}
In this setting and under the assumptions of Theorem \ref{th22}, if moreover we assume that ($QC_1$) is fulfilled for $(x,u) = (x_0,u_0)$ and (A{\sc v}3), then $(\theta_i)_{1\le i \le l}\neq 0$.
\end{corollary}
We introduce a new condition of linear independence.
\begin{itemize}
% A10
\item[\bf{(A{\sc lib})}] $U$ is a subset of a real normed vector space $Y$ s.t. $U$ is a neighborhood of $u_0(T)$ in $Y$, $D_{G,3}f(T,x_0(T),u_0(T))$ exists and \\$((D_Hg^{\alpha}(x_0(T))\circ D_{G,3}f(T,x_0(T),u_0(T)))_{1\le \alpha\le m},\\(D_Hh^{\beta}(x_0(T))\circ D_{G,3}f(T,x_0(T),u_0(T)))_{1\le \beta \le q})$ are linearly independent.
\end{itemize}
\begin{corollary}\label{cor24}
In this setting and under the assumptions of Theorem \ref{th22}, if moreover we assume (A{\sc lib}) is fulfilled, then $(\theta_i)_{1\le i \le l}\neq 0$.
\end{corollary}
For each $j\in \{1,...,l\}$, we consider the following condition:
\begin{itemize}
% A10
\item[\bf{(A{\sc f})$_j$}] $U$ is a subset of a real normed vector space $Y$ s.t. $U$ is a neighborhood of $u_0(T)$ in $Y$, $D_{G,3}f(T,x_0(T),u_0(T))$ exists and \\$((D_Hg_i^{0}(x_0(T))\circ D_{G,3}f(T,x_0(T),u_0(T)))_{i\neq j},\\(D_Hg^{\alpha}(x_0(T))\circ D_{G,3}f(T,x_0(T),u_0(T)))_{1\le \alpha\le m},\\(D_Hh^{\beta}(x_0(T))\circ D_{G,3}f(T,x_0(T),u_0(T)))_{1\le \beta \le q})$ are linearly independent.
\end{itemize}
\begin{corollary}\label{cor25}
In this setting and under the assumptions of Theorem \ref{th22}, if, for each $j\in\{1,...,l\}$, we have (A{\sc f})$_j$, then $\theta_j \neq 0$ i.e. we can take $\theta_j=1$.  
Moreover, $((\theta_i)_{1\le i\le l},(\lambda_{\alpha})_{1 \leq \alpha \leq m}, (\mu_{\beta})_{1 \leq \beta \leq q},p)\in \R^{l}\times \R^{m}\times\R^q\times PC^1([0,T], E^*)$ with $\theta_j=1$ that verify the conclusions of Theorem \ref{th22} are unique.   
\end{corollary}

\subsection{Necessary conditions of Pareto optimality for the Bolza problem}

\begin{definition}
The Hamiltonian of (${\mathcal B}$) is the function $H_B :[0,T]\times \Omega \times U\times E^*\times\R^l \rightarrow \R$ defined by, for all $(t,x,u,p, \theta)\in [0,T]\times \Omega \times U\times E^*\times\R^l$, $H_B(t,x,u,p, \theta) := \sum_{i=1}^{l} \theta_i f_i^0(t,x,u) + p \cdot f(t,x,u)$.
\end{definition}

\begin{theorem}\label{th21} (Pontryagin Principle for the Bolza problem)\\
When $(x_0,u_0)$ is a Pareto optimal solution of $(\mathcal{B})$, under (A{\sc i}1), (A{\sc i}2), (A{\sc v}1), (A{\sc v}2), (A{\sc t}1), (A{\sc t}2) and (A{\sc t}3), there exists $(\theta_i)_{1 \le i \le l}\in \R^l$, $(\lambda_{\alpha})_{1 \leq \alpha \leq m} \in \R^{m}$, $(\mu_{\beta})_{1 \leq \beta \leq q} \in \R^q$ and an adjoint function $p \in PC^1([0,T], E^*)$ which satisfy the following conditions.
\begin{itemize}
\item[(NN)] $((\theta_i)_{1 \le i\le l},(\lambda_{\alpha})_{1 \leq \alpha \leq m}$, $(\mu_{\beta})_{1 \leq \beta \leq q})\neq 0$
\item[(Si)] For all $i\in \{1,...,l\}$, $\theta_i\ge 0$ and for all $\alpha \in \{1,...,m \}$, $\lambda_{\alpha} \geq 0$.
\item[(S${\ell}$)] For all $\alpha \in \{1,...,m \}$, $\lambda_{\alpha} g^{\alpha}(x_0(T)) = 0$.
\item[(TC)] $\sum_{i=1}^{l} \theta_iD_Hg_i^0(x_0(T))+\sum_{\alpha = 1}^{m} \lambda_{\alpha} D_Hg^{\alpha}(x_0(T)) + \sum_{\beta = 1}^q \mu_{\beta} D_Hh^{\beta}(x_0(T)) = p(T)$.
\item[(AE.B)] $\underline{d}p(t) = - D_{F,2}H_B(t, x_0(t),u_0(t), p(t), (\theta_i)_{1 \le i \le l})$ for all $t \in [0,T]$.
\item[(MP.B)] For all $t \in [0,T]$, for all $\zeta \in U$, \\
$H_B(t, x_0(t), u_0(t), p(t), (\theta_i)_{1 \le i \le l}) \geq H_B(t, x_0(t), \zeta, p(t), (\theta_i)_{1 \le i \le l})$.
\item[(CH.B)] $\bar{H}_B := [ t \mapsto H_B(t, x_0(t), u_0(t), p(t), (\theta_i)_{1 \le i \le l})] \in C^0([0,T], \R)$.
\end{itemize} 
\end{theorem}
%Corollary 1.2
\begin{corollary}\label{cor12} In this setting
 and under the assumptions of Theorem \ref{th21}, if moreover we assume that ($QC_1$) is fulfilled for $(x,u) = (x_0,u_0)$, then, for all $t \in [0,T]$, $((\theta_i)_{1\le i \le l}, p(t))$ is never equal to zero.  
\end{corollary}
%%% Corollary 1.4
\begin{corollary}\label{cor13}
In this setting and under the assumptions of Theorem \ref{th21}, if moreover we assume that ($QC_1$) is fulfilled for $(x,u) = (x_0,u_0)$ and (A{\sc v}3), then $(\theta_i)_{1\le i \le l}\neq 0$.
\end{corollary}
\begin{corollary}\label{cor14}
In the setting and under the assumptions of Theorem \ref{th21}, if moreover we assume (A{\sc lib}) is fulfilled, then $(\theta_i)_{1\le i \le l}\neq 0$.
\end{corollary}
For each $j\in \{1,...,l\}$, we consider the following condition:
\begin{itemize}
\item[\bf{(A{\sc f})$^0_j$}] $U$ is a subset of a real normed vector space $Y$ s.t. $U$ is a neighborhood of $u_0(T)$ in $Y$, $D_{G,3}f(T,x_0(T),u_0(T))$ exists, $\forall i\in \{1,...,l\},\, i\neq j$ $D_{G,3}f_i^0(T,x_0(T),u_0(T))$ exists and \\$((D_Hg_i^{0}(x_0(T))\circ D_{G,3}f(T,x_0(T),u_0(T))+D_{G,3}f_i^0(T,x_0(T),u_0(T)))_{i\neq j},\\(D_Hg^{\alpha}(x_0(T))\circ D_{G,3}f(T,x_0(T),u_0(T)))_{1\le \alpha\le m},\\(D_Hh^{\beta}(x_0(T))\circ D_{G,3}f(T,x_0(T),u_0(T)))_{1\le \beta \le q})$ are linearly independent.
\end{itemize}
\begin{corollary}\label{cor15} In this setting
 and under the assumptions of Theorem \ref{th21}, if, for each $j\in\{1,...,l\}$, we have (A{\sc f})$^0_j$, then $\theta_j \neq 0$ (i.e. we can choose $\theta_j=1$).\\     
Moreover, if $D_{G,3}f_j^0(T,x_0(T),u_0(T))$ exists, then we have:\\ 
$((\theta_i)_{1\le i\le l},(\lambda_{\alpha})_{1 \leq \alpha \leq m}, (\mu_{\beta})_{1 \leq \beta \leq q},p)\in \R^{l}\times \R^{m}\times\R^q\times PC^1([0,T], E^*)$ with $\theta_j=1$ that verify the conclusions of Theorem \ref{th21} are unique.  
\end{corollary} 

\section{Sufficient conditions of Pareto optimality}
%%%%%%%%%%%%
Let $(\overline{x},\overline{u})\in PC^1([0,T],\Omega)\times NPC^0([0,T],U)$, we consider the following conditions. 
\begin{itemize}
\item[\bf{(S{\sc t}1)}] For all $i\in\{1,...,l\}$ $g^0_i$ is concave at $\overline{x}(T)$ and Hadamard differentiable at $\overline{x}(T)$.
\item[\bf{(S{\sc t}1-bis)}] For all $i\in\{1,...,l\}$ $g^0_i$ is pseudo-concave at $\overline{x}(T)$ and Hadamard differentiable at $\overline{x}(T)$.
\item[\bf{(S{\sc t}2)}] For all $\alpha \in \{1,...,m\}$, $g^\alpha$ is quasi-concave at $\overline{x}(T)$ and Hadamard differentiable at $\overline{x}(T)$.
\item[\bf{(S{\sc t}3)}] For all $\beta \in \{1,...,q\}$, $h^\beta$ and $-h^\beta$ are quasi-concave at $\overline{x}(T)$ and Hadamard differentiable at $\overline{x}(T)$.
\item[\bf{(S{\sc i}1)}] For all $i\in\{1,...,l\}$, $f_i^0\in C^0([0,T] \times \Omega \times U,\R)$.
\item[\bf{(S{\sc i}2)}] For all $t\in[0,T]$, for all $i\in\{1,...,l\}$, $D_{F,2}f_i^0(t,\overline{x}(t), \overline{u}(t))$ exists and $[t\mapsto D_{F,2}f_i^0(t,\overline{x}(t),\overline{u}(t))]\in NPC^0([0,T], E^*)$. 
\item[\bf{(S{\sc v}1)}] $f\in C^0([0,T] \times \Omega \times U,E)$.
\item[\bf{(S{\sc v}2)}] For all $t\in[0,T]$ $D_{F,2}f(t,\overline{x}(t), \overline{u}(t))$ exists and $[t\mapsto D_{F,2}f(t,\overline{x}(t),\overline{u}(t))]\in NPC^0([0,T], \mathcal{L}(E,E))$. 
% A8 et (A{\sc int}3)
\end{itemize}
\begin{theorem}\label{thsm1}
When $(\overline{x},\overline{u})\in Adm({\mathcal M})$, under (S{\sc t}1-bis), (S{\sc t}2), (S{\sc t}3), (S{\sc v}1) if there exists $((\theta_i)_{1\le i \le l}, (\lambda_\alpha)_{1\le \alpha \le m}, (\mu_\beta)_{1\le \beta \le q},p)\in \R^{l+m+q}\times PC^1([0,T],E^*)$ verifying the conclusions (NN), (Si), (S${\ell}$) and (TC) of Theorem \ref{th22} with $(x_0,u_0)=(\overline{x},\overline{u})$ and if the following condition is satisfied
\begin{itemize}
\item[\bf{(S{\sc hm}1)}] For each $(x,u)\in Adm({\mathcal M})$, for all $t\in[0,T]$ almost everywhere for the canonical measure of Borel on $[0,T]$,
$$H_M(t,\overline{x}(t), \overline{u}(t), p(t))-H_M(t,x(t),u(t),p(t)) \ge \underline{d}p(t)\cdot (x(t)-\overline{x}(t)),$$  
\end{itemize}
then we have:\\ 
if $(\theta_i)_{1\le i\le l} \neq 0$, then $(\overline{x},\overline{u})$ is a weak Pareto optimal solution of $(\mathcal{M})$,\\
if for all $i\in\{1,...,l\}$, $\theta_i \neq 0$, then $(\overline{x},\overline{u})$ is a Pareto optimal solution of $(\mathcal{M})$.
\end{theorem}
\begin{theorem}\label{thsm2}
When $(\overline{x},\overline{u})\in  Adm({\mathcal M})$, under (S{\sc t}1-bis), (S{\sc t}2), (S{\sc t}3), (S{\sc v}1), (S{\sc v}2) if there exists $((\theta_i)_{1\le i \le l}, (\lambda_\alpha)_{1\le \alpha \le m}, (\mu_\beta)_{1\le \beta \le q},p)\in \R^{l+m+q}\times PC^1([0,T],E^*)$ verifying all the conclusions of Theorem \ref{th22} with $(x_0,u_0)=(\overline{x},\overline{u})$ and if the following condition is satisfied
 \begin{itemize}
\item[\bf{(S{\sc hm}2)}] for all $(t,\xi)\in [0,T]\times \Omega$, \\
$H_M^*(t,\xi,p(t))=\max_{\zeta\in U} H_M(t,\xi,\zeta,p(t))$ exists, and for all $t\in [0,T]$ , $[\xi \mapsto H_M^*(t,\xi,p(t))]$ is concave at $\overline{x}(t)$ and G\^ateaux differentiable at $\overline{x}(t)$,
\end{itemize}   
then we have:\\ 
if $(\theta_i)_{1\le i\le l} \neq 0$, then $(\overline{x},\overline{u})$ is a weak Pareto optimal solution of $(\mathcal{M})$,\\
if for all $i\in\{1,...,l\}$, $\theta_i \neq 0$, then $(\overline{x},\overline{u})$ is a Pareto optimal solution of $(\mathcal{M})$.
\end{theorem}
%%%%%
\begin{theorem}\label{thsm3}
When $(\overline{x},\overline{u})\in  Adm({\mathcal M})$, under (S{\sc t}1-bis), (S{\sc t}2), (S{\sc t}3), (S{\sc v}1), (S{\sc v}2) if there exists $((\theta_i)_{1\le i \le l}, (\lambda_\alpha)_{1\le \alpha \le m}, (\mu_\beta)_{1\le \beta \le q},p)\in \R^{l+m+q}\times PC^1([0,T],E^*)$ verifying all the conclusions of Theorem \ref{th22} with $(x_0,u_0)=(\overline{x},\overline{u})$ and if the following condition is satisfied
\begin{itemize}
\item[\bf{(S{\sc hm}3)}] $U$ is a subset of a real normed vector space $Y$ s.t. for all $t\in [0,T]$, $U$ is a neighborhood of $\overline{u}(t)$, and for all $t\in [0,T]$, \\
$[(\xi,\zeta) \mapsto H_M(t,\xi,\zeta,p(t))$ is G\^ateaux differentiable at $(\overline{x}(t),\overline{u}(t))$ and concave at $(\overline{x}(t),\overline{u}(t))$,
\end{itemize}   
then we have:\\ 
if $(\theta_i)_{1\le i\le l} \neq 0$, then $(\overline{x},\overline{u})$ is a weak Pareto optimal solution of $(\mathcal{M})$,\\
if for all $i\in\{1,...,l\}$, $\theta_i \neq 0$, then $(\overline{x},\overline{u})$ is a Pareto optimal solution of $(\mathcal{M})$.
\end{theorem}
%%%%%
\begin{remark}
By using our constraint qualifications, we can rewrite the conclusion of Theorem \ref{thsm2} and Theorem \ref{thsm3} as follows. \\
If the condition (A{\sc lib}) or [($QC_1$) and (A{\sc v}3)] is fulfilled for $(x_0,u_0)=(\overline{x},\overline{u})$ then $(\overline{x},\overline{u})$ is a weak Pareto optimal solution of $(\mathcal{M})$,\\
if, for each $j\in\{1,...,l\}$, (A{\sc f})$_j$ is fulfilled for $(x_0,u_0)=(\overline{x},\overline{u})$, then $(\overline{x},\overline{u})$ is a Pareto optimal solution of $(\mathcal{M})$. 
\end{remark}
\begin{theorem}\label{thsb1} 
When $(\overline{x},\overline{u})\in Adm({\mathcal B})$, under (S{\sc t}1), (S{\sc t}2), (S{\sc t}3), (S{\sc i}1), (S{\sc i}2) (S{\sc v}1), (S{\sc v}2) if there exists $((\theta_i)_{1\le i \le l}, (\lambda_\alpha)_{1\le \alpha \le m}, (\mu_\beta)_{1\le \beta \le q},p)$ belongs to $\R^{l+m+q}\times PC^1([0,T],E^*)$ verifying the conclusions (NN), (Si), (S${\ell}$) and (TC) of Theorem \ref{th21} with $(x_0,u_0)=(\overline{x},\overline{u})$ and if the following condition is satisfied
\begin{itemize}
\item[\bf{(S{\sc hb}1)}] For each $(x,u)\in Adm({\mathcal B})$, for all $t\in[0,T]$ almost everywhere for the canonical measure of Borel on $[0,T]$, 
$$H_B(t,\overline{x}(t), \overline{u}(t), p(t), (\theta_i)_{1\le i\le l})-H_B(t,x(t),u(t),p(t),(\theta_i)_{1\le i\le l}) \ge \underline{d}p(t)\cdot (x(t)-\overline{x}(t)),$$  
\end{itemize}
then we have:\\ 
if $(\theta_i)_{1\le i\le l} \neq 0$, then $(\overline{x},\overline{u})$ is a weak Pareto optimal solution of $(\mathcal{B})$,\\
if for all $i\in \{1,...,l\}$, $\theta_i \neq 0 $, then $(\overline{x},\overline{u})$ is a Pareto optimal solution of $(\mathcal{B})$.
\end{theorem}
\begin{theorem}\label{thsb2}
When $(\overline{x},\overline{u})\in Adm({\mathcal B})$, under (S{\sc t}1), (S{\sc t}2), (S{\sc t}3), (S{\sc i}1), (S{\sc i}2), (S{\sc v}1), (S{\sc v}2) if there exists $((\theta_i)_{1\le i \le l}, (\lambda_\alpha)_{1\le \alpha \le m}, (\mu_\beta)_{1\le \beta \le q},p)$ belongs to $\R^{l+m+q}\times PC^1([0,T],E^*)$ verifying all the conclusions of Theorem  \ref{th21} with $(x_0,u_0)=(\overline{x},\overline{u})$ and if the following condition is satisfied
\begin{itemize}
\item[\bf{(S{\sc hb}2)}] for all $(t,\xi)\in [0,T]\times \Omega$, \\
$H_B^*(t,\xi,p(t),(\theta_i)_{1\le i \le l})=\max_{\zeta\in U} H_B(t,\xi,\zeta,p(t),(\theta_i)_{1\le i \le l})$ exists, and for all $t\in[0,T]$, $[\xi \mapsto H_B^*(t,\xi,p(t),(\theta_i)_{1\le i \le l})]$ is concave at $\overline{x}(t)$ and G\^ateaux differentiable at $\overline{x}(t)$,
\end{itemize}   
then we have:\\ 
if $(\theta_i)_{1\le i\le l} \neq 0$, then $(\overline{x},\overline{u})$ is a weak Pareto optimal solution of $(\mathcal{B})$,\\
if for all $i\in \{1,...,l\}$, $\theta_i \neq 0 $, then $(\overline{x},\overline{u})$ is a Pareto optimal solution of $(\mathcal{B})$.
\end{theorem}
%%%%%%
\begin{theorem}\label{thsb3}
When $(\overline{x},\overline{u})\in Adm({\mathcal B})$, under (S{\sc t}1), (S{\sc t}2), (S{\sc t}3), (S{\sc i}1), (S{\sc i}2), (S{\sc v}1), (S{\sc v}2) if there exists $((\theta_i)_{1\le i \le l}, (\lambda_\alpha)_{1\le \alpha \le m}, (\mu_\beta)_{1\le \beta \le q},p)$ belongs to $\R^{l+m+q}\times PC^1([0,T],E^*)$ verifying all the conclusions of Theorem  \ref{th21} with $(x_0,u_0)=(\overline{x},\overline{u})$ and if the following condition is satisfied
\begin{itemize}
\item[\bf{(S{\sc hb}3)}] $U$ is a subset of a real normed vector space $Y$ s.t. for all $t\in [0,T]$, $U$ is a neighborhood of $\overline{u}(t)$, and for all $t\in [0,T]$, \\
$[(\xi,\zeta) \mapsto H_B(t,\xi,\zeta,p(t),(\theta_i)_{1\le i \le l})$ is G\^ateaux differentiable at $(\overline{x}(t),\overline{u}(t))$ and concave at $(\overline{x}(t),\overline{u}(t))$,
\end{itemize}   
then we have:\\ 
if $(\theta_i)_{1\le i\le l} \neq 0$, then $(\overline{x},\overline{u})$ is a weak Pareto optimal solution of $(\mathcal{B})$,\\
if for all $i\in \{1,...,l\}$, $\theta_i \neq 0 $, then $(\overline{x},\overline{u})$ is a Pareto optimal solution of $(\mathcal{B})$.
\end{theorem}
\begin{remark}
By using our constraint qualifications, we can rewrite the conclusion of Theorem \ref{thsb2} and Theorem \ref{thsb3} as follows.\\ 
If the condition (A{\sc lib}) or [($QC_1$) and (A{\sc v}3)] is fulfilled for $(x_0,u_0)=(\overline{x},\overline{u})$ then $(\overline{x},\overline{u})$ is a weak Pareto optimal solution of $(\mathcal{B})$,\\
if, for each $j\in\{1,...,l\}$, (A{\sc f})$^0_j$ is fulfilled for $(x_0,u_0)=(\overline{x},\overline{u})$, then $(\overline{x},\overline{u})$ is a Pareto optimal solution of $(\mathcal{B})$.
 \end{remark}

\section{Proof of the necessary conditions} 
\subsection{Proof of the Theorem \ref{th22}}
\begin{lemma}\label{lem1} For all $i\in\{1,...,l\}$, $(x_0,u_0)$ is a solution of the following single-objective Mayer problem
\[
({\mathcal M_i})
\left\{
\begin{array}{cl}
{\rm Maximize} & J_i(x,u):=g_i^0(x(T)) \\
{\rm subject \;  to} & x \in PC^1([0,T], \Omega), u \in NPC^0([0,T], U)\\
\null & \forall t\in [0,T],\, \underline{d}x(t) = f(t,x(t), u(t)), \; x(0) = \xi_0\\
\null & \forall k \in\{ 1,...,l\},\, k\neq i, \; \; g_k^{0}(x(T)) \geq g_k^{0}(x_0(T))\\
\null & \forall \alpha \in\{ 1,...,m\}, \; \; g^{\alpha}(x(T)) \geq 0\\
\null & \forall \beta \in \{1,..., q\}, \; \; h^{\beta}(x(T)) = 0.
\end{array}\right.
\]
\end{lemma}
\begin{proof}
Let $i\in\{1,...,l\}$. We proceed by contradiction, we assume that $(x_0,u_0)$ is not a solution of $({\mathcal M_i})$ i.e. there exists $(x,u)$ an admissible process of $({\mathcal M_i})$ s.t. $g^0_i(x(T))>g_i^0(x_0(T))$.\\
This can be rewritten $(x,u)\in Adm({\mathcal M})$ s.t. $g^0_i(x(T))>g_i^0(x_0(T))$ and for all $k \in\{ 1,...,l\},\, k\neq i,$ $g_k^{0}(x(T)) \geq g_k^{0}(x_0(T))$.\\
Therefore, $(x_0,u_0)$ is not a Pareto optimal solution. This is a contradiction.
\end{proof}
For each $x\in\Omega$, for each $i\in \{2,...,l\}$, we set $\mathfrak{g}_i(x)=g_i^0(x)-g_i^0(x_0(T))$.\\     
Thanks to (A{\sc t}1), for each $i\in \{2,...,l\}$, $\mathfrak{g}_i$ is Hadamard differentiable at $x_0(T)$ and $D_H\mathfrak{g}_i(x_0(T))=D_Hg_i^0(x_0(T))$. \\
Consequently, by using the Lemma \ref{lem1} and (A{\sc t}2), (A{\sc t}3), (A{\sc v}1), (A{\sc v}2), the assumptions of Theorem 2.4 in \cite{BY1} are fulfilled for (${\mathcal M_1}$)
\[
({\mathcal M_1})
\left\{
\begin{array}{cl}
{\rm Maximize} & g_1^0(x(T)) \\
{\rm subject \;  to} & x \in PC^1([0,T], \Omega), u \in NPC^0([0,T], U)\\
\null & \forall t\in [0,T],\, \underline{d}x(t) = f(t,x(t), u(t)), \; x(0) = \xi_0\\
\null & \forall i \in\{ 2,...,l\}, \; \; {\mathfrak g}_i(x(T)) \ge 0\\
\null & \forall \alpha \in\{ 1,...,m\}, \; \; g^{\alpha}(x(T)) \geq 0\\
\null & \forall \beta \in \{1,..., q\}, \; \; h^{\beta}(x(T)) = 0.
\end{array}\right.
\]
Hence, we obtain that there exists $(\theta_i)_{1 \le i \le l}\in \R^l$, $(\lambda_{\alpha})_{1 \leq \alpha \leq m} \in \R^{m}$, $(\mu_{\beta})_{1 \leq \beta \leq q} \in \R^q$ and an adjoint function $p \in PC^1([0,T], E^*)$ which satisfy the following conditions.
\begin{itemize}
\item[(NN{\sc s})] $((\theta_i)_{1 \le i\le l},(\lambda_{\alpha})_{1 \leq \alpha \leq m}$, $(\mu_{\beta})_{1 \leq \beta \leq q})\neq 0$.
\item[(Si{\sc s})] For all $i\in \{1,...,l\}$, $\theta_i\ge 0$ and for all $\alpha \in \{1,...,m \}$, $\lambda_{\alpha} \geq 0$.
\item[(S${\ell}${\sc s})] For all $i\in\{2,...,l\}$, $\theta_i\mathfrak{g}_i(x_0(T)) = 0$ and for all $\alpha \in \{1,...,m \}$, $\lambda_{\alpha} g^{\alpha}(x_0(T)) = 0$.
\item[(TC{\sc s})] $\sum_{i=1}^{l} \theta_iD_Hg_i^0(x_0(T))+\sum_{\alpha = 1}^{m} \lambda_{\alpha} D_Hg^{\alpha}(x_0(T)) + \sum_{\beta = 1}^q \mu_{\beta} D_Hh^{\beta}(x_0(T)) = p(T)$.
\item[(AE.M{\sc s})] $\underline{d}p(t) = - D_{F,2}H_M(t, x_0(t),u_0(t), p(t))$ for all $t \in [0,T]$.
\item[(MP.M{\sc s})] For all $t \in [0,T]$, for all $\zeta \in U$, \\
$H_M(t, x_0(t), u_0(t), p(t)) \geq H_M(t, x_0(t), \zeta, p(t))$.
\item[(CH.M{\sc s})] $\bar{H}_M := [ t \mapsto H_M(t, x_0(t), u_0(t), p(t))] \in C^0([0,T], \R)$.
\end{itemize} 
Therefore, since for all $i\in\{2,...,l\}$, $\mathfrak{g}_i(x_0(T)) = 0$, (NN{\sc s}), (Si{\sc s}), (S${\ell}${\sc s}), (TC{\sc s}), (AE.M{\sc s}) , (MP.M{\sc s}) and (CH.M{\sc s}) are equivalent to (NN), (Si), (S${\ell}$), (TC), (AE.M) , (MP.M) and (CH.M).
Therefore, the proof Theorem \ref{th22} is complete.
\subsection{Proof of Corollary \ref{cor21}}
We proceed by contradiction by assuming that there exists $t_1 \in [0,T]$ such $((\theta_i)_{1\le i\le l}, p(t_1)) = (0,0)$.\\
Since (AE.M) is an homogeneous linear equation, and by using the uniqueness of the Cauchy problem ((AE.M), $p(t_1) = 0$), we obtain that $p$ is equal to zero on $[0,T]$, in particular we have $p(T) = 0$.\\
Hence, by using (TC), (Si), (S${\ell}$), (QC$_1$), we obtain that $(\forall \alpha \in\{1,...,m\}, \lambda_{\alpha} = 0)$ and $(\forall \beta \in\{1,...,q\},\, \mu_{\beta} = 0)$.\\
Therefore, since $(\theta_i)_{1\le i\le l} = 0$, we have $((\theta_i)_{1 \le i\le l},(\lambda_{\alpha})_{1 \leq \alpha \leq m}$, $(\mu_{\beta})_{1 \leq \beta \leq q})= 0$ which is a contradiction with (NN).
\subsection{Proof of Corollary \ref{cor22}}
We proceed by contradiction by assuming that there exists $t_1 \in [0,T]$ such $p(t_1) = 0$.\\
Since (AE.M) is an homogeneous linear equation, and by using the uniqueness of the Cauchy problem ((AE.M), $p(t_1) = 0$, we obtain that $p$ is equal to zero on $[0,T]$, in particular we have $p(T) = 0$.\\
Consequently, by using (TC), (Si), (S${\ell}$), (QC$_0$), we obtain that $((\theta_i)_{1 \le i\le l},(\lambda_{\alpha})_{1 \leq \alpha \leq m}$, $(\mu_{\beta})_{1 \leq \beta \leq q})= 0$ which is a contradiction with (NN).
\subsection{Proof of Corollary \ref{cor23}}
%%%%%%%%%%%%%%%%%%%%%%%%%%%%%%%%%%%%%%%%%%%%%%%%%%%%%%%%%%%%%%%%%%%%%%%%%%%%%%
We proceed by contradiction, we assume that \\$(\theta_i)_{1\le i\le l} = 0$.
Since $D_{G,3}f(\hat{t},x_0(\hat{t}), u_0(\hat{t}))$ exists, $D_{G,3}H_{M}(\hat{t},x_0(\hat{t}),u_0(\hat{t}),p(\hat{t}))$ exists and 
$$D_{G,3}H_{M}(\hat{t},x_0(\hat{t}),u_0(\hat{t}),p(\hat{t}))=p(\hat{t})\circ D_{G,3}f(\hat{t},x_0(\hat{t}),u_0(\hat{t})).$$
Therefore, by using (MP.M), we have $p(\hat{t})\circ D_{G,3}f(\hat{t},x_0(\hat{t}),u_0(\hat{t}))=0$.\\
Since $D_{G,3}f(\hat{t},x_0(\hat{t}),u_0(\hat{t}))$ is surjective, we have $p(\hat{t})=0$.\\
This is a contradiction with the Corollary \ref{cor21}, therefore $(\theta_i)_{1\le i\le l}\neq 0$.\\
\subsection{Proof the Corolloray \ref{cor24}} We proceed by contradiction, we assume that $(\theta_i)_{1\le i\le l} = 0$.\\
Since $D_{G,3}f(T,x_0(T), u_0(T))$ exists, $D_{G,3}H_{M}(T,x_0(T),u_0(T),p(T))$ exists and 
$$D_{G,3}H_{M}(T,x_0(T),u_0(T),p(T))=p(T)\circ D_{G,3}f(T,x_0(T),u_0(T)).$$
Consequently, by using (MP.M), we have $p(T)\circ D_{G,3}f(T,x_0(T),u_0(T))=0$.\\
That is why, thanks to (TC) and $(\theta_i)_{1\le i\le l} = 0$, we obtain that
\[
\left.
\begin{array}{l}   
\sum_{\alpha=1}^{m} \lambda_\alpha D_Hg^\alpha(x_0(T))\circ D_{G,3}f(T,x_0(T),u_0(T)) \\+
\sum_{\beta=1}^{q} \mu_\beta D_Hh^\beta(x_0(T))\circ D_{G,3}f(T,x_0(T),u_0(T))=0.
\end{array}
\right\}
\]
Hence, thanks to (A{\sc lib}), we have $((\lambda_\alpha)_{1\le \alpha\le m},(\mu_\beta)_{1\le \beta \le q})=0$.\\
Consequently, since $(\theta_i)_{1\le i\le l} = 0$, we have $((\theta_i)_{1 \le i\le l},(\lambda_{\alpha})_{1 \leq \alpha \leq m}$, $(\mu_{\beta})_{1 \leq \beta \leq q})= 0$ this a contradiction with (NN).
%%%%%%%%%%%%%%%
\subsection{Proof the Corolloray \ref{cor25}}Let $j\in\{1,...,l\}$. We assume that (A{\sc f})$_j$.\\
We proceed by contradiction, we assume that $\theta_j= 0$.\\
Since $D_{G,3}f(T,x_0(T), u_0(T))$ exists, $D_{G,3}H_{M}(T,x_0(T),u_0(T),p(T))$ exists and 
$$D_{G,3}H_{M}(T,x_0(T),u_0(T),p(T))=p(T)\circ D_{G,3}f(T,x_0(T),u_0(T)).$$
Consequently, by using (MP.M), we have $p(T)\circ D_{G,3}f(T,x_0(T),u_0(T))=0$.\\
That is why, thanks to (TC) and $\theta_j= 0$, we obtain that
\[
\left.
\begin{array}{l}   
\sum_{i\neq j }\theta_i D_Hg_i^0(x_0(T))\circ D_{G,3}f(T,x_0(T),u_0(T)) \\+
\sum_{\alpha=1}^{m} \lambda_\alpha D_Hg^\alpha(x_0(T))\circ D_{G,3}f(T,x_0(T),u_0(T)) \\+
\sum_{\beta=1}^{q} \mu_\beta D_Hh^\beta(x_0(T))\circ D_{G,3}f(T,x_0(T),u_0(T))=0.
\end{array}
\right\}
\]
Hence, thanks to (A{\sc f})$_j$, we have $((\theta_i)_{i\neq j}, (\lambda_\alpha)_{1\le \alpha\le m},(\mu_\beta)_{1\le \beta \le q})=0$.\\
Consequently, since $\theta_j=0$, we have $((\theta_i)_{1 \le i\le l},(\lambda_{\alpha})_{1 \leq \alpha \leq m}$, $(\mu_{\beta})_{1 \leq \beta \leq q})= 0$ this a contradiction with (NN).\\
We set $\forall i\in\{1,...,l\}$, $\theta_i'=\frac{\theta_i}{\theta_j}$, $\forall \alpha \in\{1,...,m\}$, $\lambda_\alpha':=\frac{\lambda_\alpha}{\theta_j}$, $\forall \beta \in\{1,...,q\}$, $\mu_\beta':=\frac{\mu_\beta}{\theta_j}$ and $p':=\frac{1}{\theta_j}p.$ \\
Since the set of $((\overline{\theta}_i)_{1 \le i\le l},(\overline{\lambda}_\alpha)_{1 \le \alpha \le m}, (\overline{\mu}_\beta)_{1\le \beta \le q},\overline{p})\in \R^{l+m+q}\times PC^1([0,T],E^*)$ verifying the conclusions of Theorem \ref{th22} is a cone, we have\\ $((\theta_i')_{1\le i\le l}, (\lambda_\alpha')_{1 \le \alpha \le m}, (\mu_\beta')_{1\le \beta \le q},p')$ that verifies the conclusions of Theorem \ref{th22} with $\theta'_j=1$.\\
Let $((\theta_i^1)_{1\le i\le l},(\lambda^1_\alpha)_{1\le \alpha\le m},(\mu^1_\beta)_{1\le \beta \le q},p^1)\in \R^{l+m+q}\times PC^1([0,T],E^*)$ and \\$((\theta_i^2)_{1\le i\le l},(\lambda^2_\alpha)_{1\le \alpha\le m},(\mu^2_\beta)_{1\le \beta \le q},p^2)\in \R^{l+m+q}\times PC^1([0,T],E^*)$ s.t. the conclusions of the Theorem \ref{th22} are verified with $\theta_j^1=\theta_j^2=1$.\\
Then, we have, for all $\ell \in \{1,2\}, p^\ell(T)\circ D_{G,3}f(T,x_0(T),u_0(T))=0$.
Therefore, we have $(p^1(T)-p^2(T))\circ D_{G,3}f(T,x_0(T),u_0(T))=0$. 
By using (TC), we have   
\[
\left.
\begin{array}{l} 
\sum_{i\neq j} (\theta^1_i-\theta^2_i) D_Hg_i^0(x_0(T))\circ D_{G,3}f(T,x_0(T),u_0(T)) \\+  
\sum_{\alpha=1}^{m} (\lambda^1_\alpha-\lambda^2_\alpha) D_Hg^\alpha(x_0(T))\circ D_{G,3}f(T,x_0(T),u_0(T)) \\+
\sum_{\beta=1}^{q} (\mu^1_\beta-\mu^2_\beta) D_Hh^\beta(x_0(T))\circ D_{G,3}f(T,x_0(T),u_0(T))=0.
\end{array}
\right\}
\]
Hence, by using (A{\sc f})$_j$, $\forall (i,\alpha,\beta)\in \{1,...,l\}\times \{1,...,m\}\times \{1,...,q\}$, $\theta^1_i=\theta^2_i$, $\lambda^1_\alpha=\lambda^2_\alpha$ and $\mu^1_\beta=\mu^2_\beta$.\\
Therefore, $p^1(T)=p^2(T)$; that is why (AE.M), we have : $p^1=p^2$.
\subsection{Proof of the Theorem \ref{th21}}
In \cite{BY}, by transforming the single-objective Bolza problem into a single-objective Mayer problem, the authors proof the Pontryagin Maximum Principle for the single-objective Bolza problem thanks to the Pontryagin Maximum Principle for the single-objective Mayer problem. For the proof of the Pontryagin Maximum Principle for the multiobjective Bolza problem, we will use the same reasoning. That is why, we introduce the following elements, for all $t\in [0,T]$, for all $X=(\sigma_1,...,\sigma_l,x)\in \R^l\times \Omega$, for all $u\in U$,\\
$F(t,X,u):=(f_1^0(t,x,u),...,f_l^0(t,x,u),f(t,x,u))$, $G_i^0(X):=\sigma_i+g_i^0(x)$ for all $i\in\{1,...,l\}$, $G^\alpha(X):=g^\alpha(x)$ for all $\alpha \in\{1,...,m\}$, $H^\beta(X):=h^\beta(x)$ for all $\beta\in\{1,...,q\}$.\\
Then, we can introduce the following multiobjective Mayer problem       
\[ (\mathcal{MB}) 
\left\{
\begin{array}{cl}
{\rm Maximize} & (G_1^0(X(T)),..., G^0_l(X(T)))\\
{\rm subject} \; \; {\rm to} & X \in PC^1([0,T], \R^l \times \Omega), u \in NPC^0([0,T], U)\\
\null & \underline{d}X(t) = F(t, X(t), u(t)), \; X(0) = (0, \xi_0)\\
\null & \forall \alpha \in \{ 1, ..., m\}, \; \; G^{\alpha}(X(T)) \geq 0 \\
\null & \forall \beta \in \{ 1,..., q\}, \; \; H^{\beta}(X(T)) = 0.
\end{array}
\right.
\]
\begin{lemma}\label{lem2}
For each $(x,u)\in Adm(\mathcal{B})$, by setting for all $t\in[0,T]$, for all $i\in\{1,...,l\}$, $\sigma_i(t):=\int_{0}^{t} f_i^0(s,x(s),u(s))ds$, we have $((\sigma_1,...,\sigma_l,x),u)\in Adm(\mathcal{MB})$.
\end{lemma}
\begin{proof}
Let $(x,u)\in Adm(\mathcal{B})$. Since $u\in NPC^0([0,T],U)$ and $x\in PC^1([0,T],\Omega)$, by using (A{\sc i}1), we have, for each $i\in\{1,...,l\}$, $[t\mapsto f_i^0(t,x(t),u(t))]\in NPC^0_d([0,T],\R)$.\\
Consequently, for each $i\in\{1,...,l\}$, $\sigma_i\in PC^1([0,T],\R)$ and for all $t\in [0,T]$, $\underline{d}\sigma_i(t)=f_i^0(t,x(t),u(t)).$\\
Hence, $(\sigma_1,...,\sigma_l,x)\in PC^1([0,T],\R^l\times \Omega)$ and for all $t\in[0,T]$,
\[
\begin{array}{ll}
\underline{d}(\sigma_1,...,\sigma_l,x)(t)&=(\underline{d}\sigma_1(t),...,\underline{d}\sigma_l(t),\underline{d}x(t))\\
\null &= (f_1^0(t,x(t),u(t)),...,f^0_l(t,x(t),u(t)),f(t,x(t),u(t)))\\
\null &= F(t,(\sigma_1,...,\sigma_l,x)(t),u(t))
\end{array}
\]
Moreover, we have, for all $\alpha \in\{1,...,m\}$, $G^\alpha((\sigma_1,...,\sigma_l,x)(T))=g^\alpha(x(T)) \ge 0$ and $\forall \beta \in \{1,...,q\}$, $H^\beta((\sigma_1,...,\sigma_l,x)(T))=h^\beta(x(T))=0$.\
Therefore, since $(\sigma_1,...,\sigma_l,x)(0)=(\sigma_1(0),...,\sigma_l(0),x(0))=(0,\xi_0)$, we have $((\sigma_1,...,\sigma_l,x),u)\in Adm(\mathcal{MB})$.
\end{proof}
Hence, by setting for all $i\in \{1,...,l\}$, for all $t\in [0,T]$, \\$\sigma_i^0(t):=\int_{0}^{t} f_i^0(s,x_0(s),u_0(s))ds$, by using the Lemma \ref{lem2}, we have $(X_0,u_0):=((\sigma_1^0,...,\sigma_l^0,x_0),u_0)\in Adm(\mathcal{MB})$.
\begin{lemma}\label{lem3}
$(X_0,u_0)$ is a Pareto optimal solution of the multiobjective problem ({$\mathcal MB$}).
\end{lemma}
\begin{proof}
We proceed by contradiction, we assume that $(X_0,u_0)$ is not a Pareto optimal solution for ({$\mathcal MB$}) i.e. there exists $(X,u)=((\sigma_1,...,\sigma_l,x),u)\in PC^1([0,T],\R^l\times \Omega)\times NPC^0([0,T],U)$ admissible process for (${\mathcal MB}$) s.t. for all $i\in\{1,...,l\}$,\\ $G_i^0(X(T)) \ge G_i^0(X_0(T))$ and there exists $i_0\in\{1,...,l\}$, $G_{i_0}^0(X(T)) >G_{i_0}^0(X_0(T))$.\\
Since $X\in PC^1([0,T],\R^l\times \Omega)$ and $\forall t\in [0,T]$, $\underline{d}X(t):=F(t,X(t),u(t))$, we have $x\in PC^1([0,T],\Omega)$ and for all $i\in\{1,...,l\}$, $\sigma_i\in PC^1([0,T],\R)$ s.t. $$
\forall t\in[0,T],  \, \underline{d}x(t)=f(t,x(t),u(t)) \text{ and } \underline{d}\sigma_i(t)=f_i^0(t,x(t),u(t)).
$$
Moreover, we have also for all $\alpha\in\{1,...,m\}$, $g^\alpha(x(T)) \ge 0$ and for all $\beta \in\{1,...,q\}$, $h^\beta(x(T))=0$.\\
Consequently, we have $(x,u)\in Adm({\mathcal B})$.\\
Moreover, for all $t\in[0,T]$, we have $\sigma_i(t)=\int_{0}^{t} f_i^0(s,x(s),u(s))ds$.
Then, for all $i\in\{1,...,l\}$,
\[
\begin{array}{ll}
G_i^0(X(T))&= \int_{0}^{T} f_i^0(s,x(s),u(s))ds+g_i^0(x(T))\\    
\null &\ge  G_i^0(X_0(T))=\int_{0}^{T} f_i^0(s,x_0(s),u_0(s))ds+g_i^0(x_0(T))
\end{array}
\]
and there exists $i_0\in \{1,...,l\}$,
\[
\begin{array}{ll}
G_{i_0}^0(X(T))&= \int_{0}^{T} f_{i_0}^0(s,x(s),u(s))ds+g_{i_0}^0(x(T))\\    
\null &> G_{i_0}^0(X_0(T))=\int_{0}^{T} f_{i_0}^0(s,x_0(s),u_0(s))ds+g_{i_0}^0(x_0(T)).
\end{array}
\]
This a contradiction with $(x_0,u_0)$ is a Pareto optimal solution.

\end{proof}
\begin{lemma}\label{lem3bis} The assumptions of Theorem \ref{th21} for the multiobjective Mayer problem (${\mathcal MB}$) with the Pareto optimal solution $(X_0,u_0)$ are verified.
\end{lemma}
\begin{proof}
We consider the linear functions $i\in\{1,...,l\}$, $w_i^1:\R^l \times E \rightarrow \R$ defined by, $w_i^1(\sigma_1,...,\sigma_l,\xi)=\sigma_i$ and $w^2:\R^l \times E \rightarrow E$, defined by, $w^2(\sigma_1,...,\sigma_l,\xi)=\xi$.\\
For all $i\in\{1,...,l\}$, since $G_i^0={w_i^1}_{\R^l \times \Omega} +g_i^0 \circ w^2_{|\R^l\times \Omega}$, by using the property of the chain rule of the Hadamard differentiable function, see \cite{F} p.267, and (A{\sc t}1), we have 
\begin{equation}\label{eq1}
D_HG^0((\sigma_1^0,...,\sigma_l^0,x_0)(T))=w_i^1+D_Hg_i^0(x_0(T))\circ w^2.
\end{equation}
Therefore, (A{\sc t}1) is verified for (${\mathcal MB}$) with the Pareto optimal solution $(X_0,u_0)$.\\
Next, for all $\alpha\in\{1,...,m\}$, since $G^\alpha=g^\alpha \circ w^2_{|\R^l\times \Omega}$, by using the property of the chain rule of the Hadamard differentiable function, see \cite{F} p.267, and (A{\sc t}2), we have 
\begin{equation}\label{eq2}
D_HG^\alpha((\sigma_1^0,...,\sigma_l^0,x_0)(T))=D_Hg^\alpha(x_0(T))\circ w^2.
\end{equation}
Hence, (A{\sc t}2) is verified for (${\mathcal MB}$) with the Pareto optimal solution $(X_0,u_0)$.\\
Moreover, for all $\beta\in\{1,...,q\}$, since $H^\beta=h^\beta \circ w^2_{|\R^l\times \Omega}$, by using the property of chain rule of the Hadamard differentiable function, see \cite{F} p.267, and (A{\sc t}3), we have 
\begin{equation}\label{eq3}
D_HH^\beta((\sigma_1^0,...,\sigma_l^0,x_0)(T))=D_Hh^\beta(x_0(T))\circ w^2.
\end{equation}
Since $h^\beta$ is continuous on a neighborhood $V_0^\beta$ of $x_0(T)$ in $\Omega$ and $w^2_{|\R^l\times \Omega}\in C^0(\R^l\times \Omega,\Omega)$, there exists $W_0^\beta$ of $X_0(T)$ in $\R^l\times \Omega$ s.t. ${w_2}_{|W_\beta^0}\in C^0(W_0^\beta,V_0^\beta)$. Hence, we have $H^\beta_{|W_0^\beta}\in C^0(W_0^\beta,\R)$.
Consequently, (A{\sc t}3) is verified for (${\mathcal MB}$) with the Pareto optimal solution $(X_0,u_0)$.\\
We consider the continuous function $\chi:[0,T]\times \R^l \times \Omega \times U \rightarrow [0,T]\times \R^l \times \Omega$ defined by $\chi(t,\sigma,\xi,\zeta)=(t,\xi,\zeta)$.\\
We remark that $F:=(f^0_1 \circ \chi,...,f^0_l\circ \chi,f\circ \chi)$.\\
By using (A{\sc i}1) and (A{\sc v}1), we have, for all $i\in\{1,...,l\}$, $f^0_1 \circ \chi\in C^0( [0,T]\times \R^l \times \Omega \times U, \R)$ and $f \circ \chi\in C^0( [0,T]\times \R^l \times \Omega \times U, E)$.\\
Consequently, we have $F\in C^0( [0,T]\times \R^l \times \Omega \times U, \R^l\times E)$.\\
By using (A{\sc i}1) and (A{\sc v}1), we have, for all $(t,\sigma,\xi,\zeta)\in [0,T]\times \R^l \times \Omega \times U$, $D_{G,2}F(t,(\sigma,\xi),\zeta)$ exists and 
\begin{equation}\label{eq4}
\left.
\begin{array}{ll}
D_{G,2} F(t,(\sigma,\xi),\zeta)\\
=(D_{G,2}f_1^0(t,\xi,\zeta)\circ w^2,..., D_{G,2}f_l^0(t,\xi,\zeta)\circ w^2, D_{G,2}f(t,\xi,\zeta)\circ w^2).
\end{array}
\right\}
\end{equation}
For all $t\in[0,T]$ and $\zeta\in U$, since $F(t,\cdot,\zeta):=(f_1^0(t,\cdot,\zeta)\circ w_{|\R^l\times \Omega}^2,...,f_l^0(t,\cdot,\zeta)\circ w_{|\R^l\times \Omega}^2,f(t,\cdot,\zeta)\circ w_{|\R^l\times \Omega}^2 )$, by using (A{\sc i}1) and (A{\sc v}1), we have $D_{F,2}F(t,X_0(t),\zeta)$ exists and 
$$
\left.
\begin{array}{ll}
D_{F,2}F(t,X_0(t),\zeta)\\
=(D_{F,2}f_1^0(t,x_0(t),\zeta)\circ w^2,..., D_{F,2}f_l^0(t,x_0(t),\zeta)\circ w^2, D_{F,2}f(t,x_0(t),\zeta)\circ w^2).
\end{array}
\right\}$$
Consequently, by using (A{\sc i}1) and (A{\sc v}1), we have $$[(t,\zeta) \mapsto D_{F,2}F(t,X_0(t),\zeta)]\in C^0([0,T] \times U,\mathcal{L}(\R^l\times E,\R^l\times E).$$ 
Therefore, (A{\sc v}1) is verified for (${\mathcal MB}$) with the Pareto optimal solution $(X_0,u_0)$.\\
Let $K$ be a non-empty compact s.t. $K \subset \R^l\times \Omega$ and $M$ be a non-empty compact s.t. $M\subset U$.\\
We consider the linear continuous function $\varpi :\R^l\times \Omega \rightarrow \Omega$, defined by, for all $(\sigma,\xi)\in \R^l\times \Omega$, $\varpi(\sigma,\xi):=\xi$.\\
Since $K$ is a non-empty compact, $\tilde{K}=\varpi(K)$ is a non empty compact s.t. $\tilde{K} \subset \Omega$.\\
Consequently, by using (A{\sc i}2) and (A{\sc v}2), we have $$
\text{for all } i\in \{1,...,l\} \sup_{(t,\xi,\zeta)\in [0,T]\times \tilde{K}\times M} \|D_{G,2}f_i^0(t,\xi,\zeta)\|_\mathcal{L}<+\infty,$$
and $$\sup_{(t,\xi,\zeta)\in [0,T]\times \tilde{K}\times M} \|D_{G,2}f(t,\xi,\zeta)\|_\mathcal{L}<+\infty.$$  
Therefore, by using (\ref{eq4}), we have $$
\begin{array}{l}
\underset{(t,(\sigma,\xi),\zeta)\in [0,T]\times K\times U}{\sup} \|D_{G,2} F(t,(\sigma,\xi),\zeta)\|_\mathcal{L} \\
\le \sum_{i=1}^{l}  \underset{(t,\xi,\zeta)\in [0,T]\times \tilde{K}\times M}{\sup} \|D_{G,2}f_i^0(t,\xi,\zeta)\|_\mathcal{L}+\underset{(t,\xi,\zeta)\in [0,T]\times \tilde{K}\times M}{\sup} \|D_{G,2}f(t,\xi,\zeta)\|_\mathcal{L}\\
<+\infty.
\end{array}
$$
Hence (A{\sc v}2) is verified for (${\mathcal MB}$) with the Pareto optimal solution $(X_0,u_0)$.
\end{proof} 
By using the Lemma \ref{lem3bis}, by applying the Theorem \ref{th21}, we obtain that, there exists $(\theta_i)_{1 \le i \le l}\in \R^l$, $(\lambda_{\alpha})_{1 \leq \alpha \leq m} \in \R^{m}$, $(\mu_{\beta})_{1 \leq \beta \leq q} \in \R^q$ and an adjoint function $P \in PC^1([0,T], (\R^l\times E)^*)$ which satisfy the following conditions.
\begin{itemize}
\item[(i)] $((\theta_i)_{1 \le i\le l},(\lambda_{\alpha})_{1 \leq \alpha \leq m}$, $(\mu_{\beta})_{1 \leq \beta \leq q})\neq 0$
\item[(ii)] For all $i\in \{1,...,l\}$, $\theta_i\ge 0$ and for all $\alpha \in \{1,...,m \}$, $\lambda_{\alpha} \geq 0$.
\item[(iii)] For all $\alpha \in \{1,...,m \}$, $\lambda_{\alpha} G^{\alpha}(X_0(T)) = 0$.
\item[(iv)] $\sum_{i=1}^{l} \theta_iD_HG_i^0(X_0(T))+\sum_{\alpha = 1}^{m} \lambda_{\alpha} D_HG^{\alpha}(X_0(T)) + \sum_{\beta = 1}^q \mu_{\beta} D_HH^{\beta}(X_0(T)) = P(T)$.
\item[(v)] $\underline{d}P(t) = - D_{F,2}H_{MB}(t, X_0(t),u_0(t), P(t))$ for all $t \in [0,T]$.
\item[(vi)] For all $t \in [0,T]$, for all $\zeta \in U$, \\
$H_{MB}(t, X_0(t), u_0(t), P(t)) \geq H_{MB}(t, X_0(t), \zeta, P(t))$.
\item[(vii)] $\bar{H}_{MB} := [ t \mapsto H_{MB}(t, X_0(t), u_0(t), P(t))] \in C^0([0,T], \R)$.
\end{itemize} 
Where the function $H_{MB}:[0,T]\times (\R^l\times \Omega) \times U\times (\R^l\times E)^* \rightarrow \R$ is the Hamiltonian of the problem (${\mathcal MB}$), defined by $H_{MB}(t, (\sigma_1,...,\sigma_l,x),u,P)=P\cdot F(t,(\sigma_1,...,\sigma_l,x),u).$\\
We consider the linear continuous function $\psi  :(\R^l\times E)^*\rightarrow E ^*$ defined by, for all ${\mathfrak l}\in(\R^l\times E)^*$, for all $\xi\in E$, $\psi({\mathfrak l})\cdot={\mathfrak l}\cdot(0,\xi)$.\\
We set $p=\psi\circ P$. Since $\psi\in\mathcal{L}((\R^l\times E)^*,E^*)$, we have $p\in PC^1([0,T],E^*)$ and for all $t\in [0,T]$, $\underline{d}p(t)=\psi\cdot\underline{d}P(t).$\\
Therefore, by using (i), (ii) and (iii), we have respectively (NN), (Si) and (S${\ell}$).\\
By using (iv), we have , for each $i\in\{1,...,l\}$ $P(T)\cdot (e_i,0)=\theta_i$ where $(e_i)_{1\le i\le l}$ is the canonical basis of $\R^l$ and $\forall \xi\in E$, 
$$
\begin{array}{l}
p(T)\cdot \xi=P(T)\cdot(0,\xi)\\
=(\sum_{i=1}^{l} \theta_iD_HG_i^0(X_0(T))+\sum_{\alpha = 1}^{m} \lambda_{\alpha} D_HG^{\alpha}(X_0(T)) \\
+ \sum_{\beta = 1}^q \mu_{\beta} D_HH^{\beta}(X_0(T)))\cdot(0,\xi)\\
= (\sum_{i=1}^{l} \theta_iD_Hg_i^0(x_0(T)) +\sum_{\alpha = 1}^{m} \lambda_{\alpha} D_Hg^{\alpha}(x_0(T))  + \sum_{\beta = 1}^q \mu_{\beta} D_Hh^{\beta}(x_0(T)))\cdot \xi.
\end{array}
$$
Hence (TC) is verified.\\
For all $i\in\{1,...,l\}$, we consider the linear continuous function $\varphi_i :(\R^l\times E)^* \rightarrow \R$ defined by, $\forall {\mathfrak l}\in (\R^l\times E)^*,$ $\varphi_i({\mathfrak l})={\mathfrak l}\cdot(e_i,0)$. We set $p_0^i=\varphi_i \circ P$.\\
Since $\varphi_i\in \mathcal{L}((\R^l\times E)^*,\R)$ we have $p^0_i\in PC^1([0,T],\R)$ and 
$$\underline{d}p_0^i(t)=\varphi_i \cdot \underline{d}P(t)=\underline{d}P(t)\cdot(e_i,0)=0.$$
Moreover, since $p_0^i(T)=\theta_i$, we have $\forall t\in[0,T]$, $\underline{d}p_i^0(t)=\theta_i$.\\
Besides, $\forall \xi\in E$, $\forall t\in[0,T]$,
$$ 
\begin{array}{ll}
\underline{d}p(t)\cdot \xi&=\underline{d}P(t)\cdot(0,\xi) \\
\null & =-P(t)\cdot D_{F,2}F(t,X_0(t),u_0(t))\cdot (0,\xi)\\
\null &= -\sum_{i=1}^{l} p_0^i(t)D_{F,2}f_i^0(t,x_0(t),u_0(t))\cdot \xi-p(t)\cdot D_{F,2}f(t,x_0(t),u_0(t))\cdot \xi\\
\null &= -\sum_{i=1}^{l} \theta_iD_{F,2}f_i^0(t,x_0(t),u_0(t))\cdot \xi-p(t)\cdot D_{F,2}f(t,x_0(t),u_0(t))\cdot \xi\\
\null &=-D_{F,2}H_{B}(t, x_0(t),u_0(t), p(t), (\theta_i)_{1 \le i \le l})\cdot \xi
\end{array}
$$
Consequently (AE.B) is verified.\\
Furthermore, we have, $\forall (t,\zeta)\in [0,T]\times U$, 
$$ 
H_{MB}(t, X_0(t),\zeta, P(t))=H_{B}(t, x_0(t),\zeta, p(t), (\theta_i)_{1 \le i \le l}).$$
Consequently, by using (vi) and (vii), we have proved (MP.B) and (CH.B).\\
Hence the proof of the Theorem \ref{th21} is complete.
%%%%%%%%%%%%%%%%%%%%%%%%%%%%%%%%%%%%%%%%%%%%%%%%%%%%%%%%%%%%%%%%%%%%%%%%%%%%%%%%%%%%%%%%%%%%%%%%%%%%%%%%%%%%%%%%%%
\subsection{Proof of Corollary \ref{cor12}}
We proceed by contradiction by assuming that there exists $t_1 \in [0,T]$ such $((\theta_i)_{1\le i\le l}, p(t_1)) = (0,0)$.\\
Since (AE.B) becomes an homogeneous linear equation, and by using the uniqueness of the Cauchy problem ((AE.B), $p(t_1) = 0$), we obtain that $p$ is equal to zero on $[0,T]$, in particular we have $p(T) = 0$.\\
Hence, by using (TC), (Si), (S${\ell}$), (QC$_1$), we obtain that $(\forall \alpha \in\{1,...,m\}, \lambda_{\alpha} = 0)$ and $(\forall \beta \in\{1,...,q\},\, \mu_{\beta} = 0)$.\\
Therefore, since $(\theta_i)_{1\le i\le l} = 0$, we have $((\theta_i)_{1 \le i\le l},(\lambda_{\alpha})_{1 \leq \alpha \leq m}$, $(\mu_{\beta})_{1 \leq \beta \leq q})= 0$ which is a contradiction with (NN).
\subsection{Proof of Corollary \ref{cor13}}
%%%%%%%%%%%%%%%%%%%%%%%%%%%%%%%%%%%%%%%%%%%%%%%%%%%%%%%%%%%%%%%%%%%%%%%%%%%%%%
We proceed by contradiction, we assume that $(\theta_i)_{1\le i\le l} = 0$.\\
Since $D_{G,3}f(\hat{t},x_0(\hat{t}), u_0(\hat{t}))$ exists, $D_{G,3}H_{B}(\hat{t},x_0(\hat{t}),u_0(\hat{t}),p(\hat{t}),0)$ exists and 
$$D_{G,3}H_{B}(\hat{t},x_0(\hat{t}),u_0(\hat{t}),p(\hat{t}),0)=p(\hat{t})\circ D_{G,3}f(\hat{t},x_0(\hat{t}),u_0(\hat{t})).$$
Therefore, by using (MP.B), we have $p(\hat{t})\circ D_{G,3}f(\hat{t},x_0(\hat{t}),u_0(\hat{t}))=0$.\\
Since $D_{G,3}f(\hat{t},x_0(\hat{t}),u_0(\hat{t}))$ is surjective, we have $p(\hat{t})=0$.\\
Therefore, we have $((\theta_i)_{1\le i\le l},p(\hat{t}))=0$.\\
This is a contradiction with the Corollary \ref{cor12}, therefore $(\theta_i)_{1\le i\le l}\neq 0$.\\
\subsection{Proof the Corolloray \ref{cor14}} We proceed by contradiction, we assume that $(\theta_i)_{1\le i\le l} = 0$.\\
Since $D_{G,3}f(T,x_0(T), u_0(T))$ exists, $D_{G,3}H_{B}(T,x_0(T),u_0(T),p(T),0)$ exists and 
$$D_{G,3}H_{M}(T,x_0(T),u_0(T),p(T),0)=p(T)\circ D_{G,3}f(T,x_0(T),u_0(T)).$$
Consequently, by using (MP.B), we have $p(T)\circ D_{G,3}f(T,x_0(T),u_0(T))=0$.\\
That is why, thanks to (TC) and $(\theta_i)_{1\le i\le l} = 0$, we obtain that
\[
\left.
\begin{array}{l}   
\sum_{\alpha=1}^{m} \lambda_\alpha D_Hg^\alpha(x_0(T))\circ D_{G,3}f(T,x_0(T),u_0(T)) \\+
\sum_{\beta=1}^{q} \mu_\beta D_Hh^\beta(x_0(T))\circ D_{G,3}f(T,x_0(T),u_0(T))=0.
\end{array}
\right\}
\]
Hence, thanks to (A{\sc lib}), we have $((\lambda_\alpha)_{1\le \alpha\le m},(\mu_\beta)_{1\le \beta \le q})=0$.\\
Consequently, since $(\theta_i)_{1\le i\le l} = 0$, we have $((\theta_i)_{1 \le i\le l},(\lambda_{\alpha})_{1 \leq \alpha \leq m}$, $(\mu_{\beta})_{1 \leq \beta \leq q})= 0$ this a contradiction with (NN).
\subsection{Proof the Corolloray \ref{cor15}}Let $j\in\{1,...,l\}$. We assume that (A{\sc f})$_j$.\\
We proceed by contradiction, we assume that $\theta_j= 0$.\\
Since $D_{G,3}f(T,x_0(T), u_0(T))$ exists and for all $i \neq j$, $D_{G,3}f^0_i(T,x_0(T),u_0(T))$ exists,\\ $D_{G,3}H_{B}(T,x_0(T),u_0(T),p(T),(\theta_i)_{1\le i\le l})$ exists and 
$$
\left.
\begin{array}{ll}
D_{G,3}H_{B}(T,x_0(T),u_0(T),p(T),(\theta_i)_{1\le i\le l})\\
=p(T)\circ D_{G,3}f(T,x_0(T),u_0(T))+\sum_{i\neq j}\theta_iD_{G,3}f^0_i(T,x_0(T),u_0(T)).
\end{array}
\right\}
$$
Consequently, by using (MP.B), we have $$p(T)\circ D_{G,3}f(T,x_0(T),u_0(T))+\sum_{i\neq j}\theta_iD_{G,3}f^0_i(T,x_0(T),u_0(T))=0.$$
That is why, thanks to (TC) and $\theta_j= 0$, we obtain that
\[
\left.
\begin{array}{l}   
\sum_{i\neq j }\theta_i (D_Hg_i^0(x_0(T))\circ D_{G,3}f(T,x_0(T),u_0(T))+D_{G,3}f^0_i(T,x_0(T),u_0(T))) \\+
\sum_{\alpha=1}^{m} \lambda_\alpha D_Hg^\alpha(x_0(T))\circ D_{G,3}f(T,x_0(T),u_0(T)) \\+
\sum_{\beta=1}^{q} \mu_\beta D_Hh^\beta(x_0(T))\circ D_{G,3}f(T,x_0(T),u_0(T))=0.
\end{array}
\right\}
\]
Hence, thanks to (A{\sc f})$^0_j$, we have $((\theta_i)_{i\neq j}, (\lambda_\alpha)_{1\le \alpha\le m},(\mu_\beta)_{1\le \beta \le q})=0$.\\
Consequently, since $\theta_j=0$, we have $((\theta_i)_{1 \le i\le l},(\lambda_{\alpha})_{1 \leq \alpha \leq m}$, $(\mu_{\beta})_{1 \leq \beta \leq q})= 0$ this a contradiction with (NN).\\
We set $\forall i\in\{1,...,l\}$, $\theta_i'=\frac{\theta_i}{\theta_j}$, $\forall \alpha \in\{1,...,m\}$, $\lambda_\alpha':=\frac{\lambda_\alpha}{\theta_j}$, $\forall \beta \in\{1,...,q\}$, $\mu_\beta':=\frac{\mu_\beta}{\theta_j}$ and $p':=\frac{1}{\theta_j}p.$ \\
Since the set of $((\overline{\theta}_i)_{1 \le i\le l},(\overline{\lambda}_\alpha)_{1 \le \alpha \le m}, (\overline{\mu}_\beta)_{1\le \beta \le q},\overline{p})\in \R^{l+m+q}\times PC^1([0,T],E^*)$ verifying the conclusions of Theorem \ref{th21} is a cone, we have\\ $((\theta_i')_{1\le i\le l}, (\lambda_\alpha')_{1 \le \alpha \le m}, (\mu_\beta')_{1\le \beta \le q},p')$ that verifies the conclusions of Theorem \ref{th21} with $\theta'_j=1$.\\
Now, we assume that $D_{G,3}f^0_j(T,x_0(T),u_0(T))$ exists.\\ 
Let $((\theta_i^1)_{1\le i\le l},(\lambda^1_\alpha)_{1\le \alpha\le m},(\mu^1_\beta)_{1\le \beta \le q},p^1)\in \R^{l+m+q}\times PC^1([0,T],E^*)$ and \\$((\theta_i^2)_{1\le i\le l},(\lambda^2_\alpha)_{1\le \alpha\le m},(\mu^2_\beta)_{1\le \beta \le q},p^2)\in \R^{l+m+q}\times PC^1([0,T],E^*)$ s.t. the conclusions of the Theorem \ref{th21} are verified with $\theta_j^1=\theta_j^2=1$.\\
Then, we have, for all $\ell \in \{1,2\}$,
$$
\left.
\begin{array}{l}
p^\ell(T)\circ D_{G,3}f(T,x_0(T),u_0(T))+D_{G,3}f_j^0(T,x_0(T),u_0(T))+ \\
\sum_{i\neq j}\theta^\ell_iD_{G,3}f_i^0(T,x_0(T),u_0(T)) =0.
\end{array}
\right\}
$$
Therefore, we have 
$$(p^1(T)-p^2(T))\circ D_{G,3}f(T,x_0(T),u_0(T))+\sum_{i\neq j}(\theta_i^1-\theta_i^2)D_{G,3}f_i^0(T,x_0(T),u_0(T))=0.$$
By using (TC), we have   
\[
\left.
\begin{array}{l} 
\sum_{i\neq j} (\theta^1_i-\theta^2_i) (D_Hg_i^0(x_0(T))\circ D_{G,3}f(T,x_0(T),u_0(T))+\\D_{G,3}f_i^0(T,x_0(T),u_0(T))) \\+  
\sum_{\alpha=1}^{m} (\lambda^1_\alpha-\lambda^2_\alpha) D_Hg^\alpha(x_0(T))\circ D_{G,3}f(T,x_0(T),u_0(T)) \\+
\sum_{\beta=1}^{q} (\mu^1_\beta-\mu^2_\beta) D_Hh^\beta(x_0(T))\circ D_{G,3}f(T,x_0(T),u_0(T))=0.
\end{array}
\right\}
\]
Hence, by using (A{\sc f})$^0_j$, $\forall (i,\alpha,\beta)\in \{1,...,l\}\times \{1,...,m\}\times \{1,...,q\}$, $\theta^1_i=\theta^2_i$, $\lambda^1_\alpha=\lambda^2_\alpha$ and $\mu^1_\beta=\mu^2_\beta$.\\
Therefore, $p^1(T)=p^2(T)$; that is why, by using (AE.B), we have : $p^1=p^2$.
%%%%%%%%%%%%%%%%%%%%%%%%%%%%%%%%%%%%%%%%%%%%%%%%%%%%%%%%%%%%%%%%%%%%%%%%%%%%%%%%%%%%%%%%%%%%%%%%%%%%%%%%%%%%%%%%%%%%%%%%%%%%%%%%%%%%%%%%%%%%%%%%%%%%%%%%%%%%%%%%%%%%%%%%%%%%%%%%%%%%%%%%%%%%%%%%%%%%%%%%%%%%%%%%%%%%%%%%%%%%%%%%%%%%%%%%%%%%%%%%%%%%%%%%%%%%%%%%%%%%%%%%%%%%%%%%%%%%%%%%
\section{Proof of the sufficient conditions} 
\subsection{Proof of the Theorem \ref{thsm1}}
Let $(x,u)\in Adm({\mathcal M})$. By using (TC), we have 
\begin{equation}
\left.
\begin{array}{l}
\sum_{i=1}^{l} \theta_iD_Hg^0_i(\overline{x}(T))\cdot(x(T)-\overline{x}(T))\\
=p(T)\cdot (x(T)-\overline{x}(T))-\sum_{\alpha=1}^{m} \lambda_\alpha D_Hg^\alpha(\overline{x}(T)) \cdot (x(T)-\overline{x}(T))\\
-\sum_{\beta=1}^{q} \mu_\beta D_Hh^\beta(\overline{x}(T))\cdot(x(T)-\overline{x}(T))
\end{array}
\right\}
\end{equation}
Moreover, by using (Si) and (S${\ell}$), we have for each $\alpha \in\{1,...,m\}$, $\lambda_\alpha g^\alpha(\overline{x}(T)) \le \lambda_\alpha g^\alpha(x(T))$.\\
Consequently, by using (S{\sc t}2), we have for all $\alpha \in\{1,...,m\}$, $\lambda_\alpha D_Hg^\alpha(\overline{x}(T)) \cdot(x(T)-\overline{x}(T))\ge 0$.\\
Moreover, thanks to (S{\sc t}3), we have for all $\beta\in\{1,...,q\}$, $\mu_\beta D_Hh^\beta(\overline{x}(T))\cdot(x(T)-\overline{x}(T))=0$.\\
Hence
$$
\begin{array}{l}
\sum_{i=1}^{l}\theta_i D_Hg^0_i(\overline{x}(T))\cdot(x(T)-\overline{x}(T))\\
\le p(T)\cdot (x(T)-\overline{x}(T)) \\
=p(0)\cdot (x(0)-\overline{x}(0))+\int_{0}^{T} \underline{d}(p(t)\cdot(x(t)-\overline{x}(t)))dt\\
=\int_{0}^{T} \underline{d}p(t)\cdot (x(t)-\overline{x}(t)) dt + \int_{0}^{T} p(t)\cdot \underline{d}(x(t)-\overline{x}(t)) dt\\
\le \int_{0}^{T} (H_M(t,\overline{x}(t),\overline{u}(t),p(t))-H_M(t,x(t),u(t),p(t)))dt+\\
\int_{0}^{T} (H_M(t,x(t),u(t),p(t))-H_M(t,\overline{x}(t),\overline{u}(t),p(t)))dt\\
=0 
\end{array}
$$
where we have used (S{\sc hm}1).

Therefore, thanks to (S{\sc t}1-bis), we have $\sum_{i=1}^{l} \theta_i g^0_i(x(T)) \le \sum_{i=1}^{l} \theta_i g^0_i(\overline{x}(T))$.\\
Hence, $(\overline{x},\overline{u})$ is a solution of the following single-objective optimization problem :
\[
(\mathcal{P}_\theta)
\left\{
\begin{array}{cl}
{\rm Maximize} & \sum_{i=1}^{l} \theta_i J_i(x,u) \\
{\rm subject \;  to} & (x,u)\in Adm({\mathcal M}).
\end{array}
\right.
\]
Now, we assume that $(\theta_i)_{1\le i\le l} \neq 0$. \\
We want to prove that $(\overline{x},\overline{u})$ is a weak Pareto optimal solution. We proceed by contradiction, we assume that $(\overline{x},\overline{u})$ is not a weak Pareto  optimal solution i.e. there exists $(x,u)\in Adm({\mathcal M})$ such that for all $i\in \{1,...,l\}$, $J_i(x,u) > J_i(\overline{x},\overline{u})$.\\
Consequently, we have $\sum_{i=1}^{l} \theta_i J_i(x,u) > \sum_{i=1}^{l} \theta_i J_i(\overline{x},\overline{u})$. But this contradicts the optimality of $(\overline{x},\overline{u})$ for the problem $({\mathcal P}_\theta)$. \\
Next, we assume that for all $i\in\{1,...,l\},$ $\theta_i \neq 0$.\\
We want to prove that $(\overline{x},\overline{u})$ is a Pareto optimal solution. We proceed by contradiction, we assume that $(\overline{x},\overline{u})$ is not a Pareto optimal solution i.e. there exists $(x,u)\in Adm({\mathcal M})$ such that for all $i\in \{1,...,l\}$, $J_i(x,u)\ge J_i(\overline{x},\overline{u})$ and for some $i_0\in\{1,...,l\}$, $J_{i_0}(x,u)>J_{i_0}(\overline{x},\overline{u})$.
Hence, we obtain that $\sum_{i=1}^{l} \theta_i J_i(x,u) > \sum_{i=1}^{l} \theta_i J_i(\overline{x},\overline{u})$ which contradicts the optimality of $(\overline{x},\overline{u})$.

\subsection{Proof of the Theorem \ref{thsm2}}
Notice that (S{\sc hm}2) implies (S{\sc hm}1).\\
Indeed, let $(x,u)\in Adm({\mathcal M})$.\\
For all $t\in [0,T]$, for all $\varepsilon >0$ small enough, we have $\overline{x}(t)+\varepsilon(\overline{x}(t)-x(t)) \in \Omega$, 
therefore by using (MP.M) 
$$
\begin{array}{l}
\frac{1}{\varepsilon} (H_M^*(t,\overline{x}(t)+\varepsilon(\overline{x}(t)-x(t)),p(t))-H_M^*(t,\overline{x}(t),p(t))) \\
\ge \frac{1}{\varepsilon} (H_M(t,\overline{x}(t)+\varepsilon(\overline{x}(t)-x(t)),\overline{u}(t),p(t))-H_M(t,\overline{x}(t),\overline{u}(t),p(t))).
\end{array}
$$ 
Therefore, since (S{\sc hm}2) and (S{\sc v}2), when $\varepsilon \rightarrow 0$ we have $D_{G,2}H_M^*(t,\overline{x}(t),p(t))\cdot (\overline{x}(t)-x(t)) \ge D_{G,2}H_M(t,\overline{x}(t),\overline{u}(t),p(t))\cdot (\overline{x}(t)-x(t))$.
Therefore, by using (AE.M), we have 
\begin{equation}\label{sm1-1}
-D_{G,2}H_M^*(t,\overline{x}(t),p(t))\cdot (x(t)-\overline{x}(t)) \ge \underline{d}p(t)\cdot (x(t)-\overline{x}(t)).
\end{equation} 
Besides, for all $\varepsilon >0$ small enough, we have $\overline{x}(t)+\varepsilon(x(t)-\overline{x}(t)) \in \Omega$, 
therefore by using (MP.M) and (S{\sc hm}2), we have
$$ 
\begin{array}{l}
\frac{1}{\varepsilon} (H_M^*(t,\overline{x}(t)+\varepsilon(x(t)-\overline{x}(t)),p(t))-H_M^*(t,\overline{x}(t),p(t))) \\
\ge H_M^*(t,x(t),p(t))-H_M^*(t,\overline{x}(t),p(t))\\
\ge H_M(t,x(t),u(t),p(t))-H_M(t,\overline{x}(t),\overline{u}(t),p(t)).
\end{array}
$$
Hence, we have
$$
\begin{array}{l}
H_M(t,\overline{x}(t),\overline{u}(t),p(t))-H_M(t,x(t),u(t),p(t))\\
\ge \frac{1}{\varepsilon} (H_M^*(t,\overline{x}(t),p(t))-H_M^*(t,\overline{x}(t)+\varepsilon(x(t)-\overline{x}(t)),p(t))).
\end{array}
$$
Consequently, when $\varepsilon \rightarrow 0$ and thanks to (AE.M) and (\ref{sm1-1}), we have \\
$H_M(t,\overline{x}(t),\overline{u}(t),p(t))-H_M(t,x(t),u(t),p(t)) \ge \underline{d}p(t)\cdot(x(t)-\overline{x}(t))$.\\
Hence, the assumptions of the Theorem \ref{thsm1} are verified and the conclusions follow.\\
%%%%%%%%%%%%

\subsection{Proof of the Theorem \ref{thsm3}}
Notice that (S{\sc hm}3) implies (S{\sc hm}1).\\
Indeed, let $(x,u)\in Adm({\mathcal M})$, let $t\in [0,T]$, since $[(\xi,\zeta)\mapsto H_M(t,\xi,\zeta,p(t))]$ is G\^ateaux differentiable and concave at $(\overline{x}(t),\overline{u}(t))$, we have $H_M(t,x(t),u(t),p(t))-H_M(t,\overline{x}(t),\overline{u}(t),p(t))\\
\le D_{G,(2,3)}H_M(t,\overline{x}(t),\overline{u}(t),p(t))\cdot(x(t)-\overline{x}(t),u(t)-\overline{u}(t))$.
Therefore, by using (AE.M) and (MP.M), we have \\
$D_{G,(2,3)}H_M(t,\overline{x}(t),\overline{u}(t),p(t))\cdot(x(t)-\overline{x}(t),u(t)-\overline{u}(t))=-\underline{d}p(t)\cdot(x(t)-\overline{x}(t))$.   
Therefore, (S{\sc hm}1) is verified. Hence, the assumptions of the Theorem \ref{thsm1} are verified and the conclusions follow. %%%%%%%%%%%%%%%
\subsection{Proof of the Theorem \ref{thsb1}}
Let $(x,u)\in Adm({\mathcal B})$. By using (S{\sc T}1), we have 
\[
\begin{array}{l}
\sum_{i=1}^{l}\theta_i J_i(x,u)=\sum_{i=1}^{l} \theta_ig_i^0(x(T))+\int_{0}^{T} \sum_{i=1}^{l} \theta_i f^0_i(t,x(t),u(t))dt \\
\le \sum_{i=1}^{l} \theta_ig_i^0(\overline{x}(T))+\sum_{i=1}^{l}\theta_iD_Hg_i(\overline{x}(T))\cdot(x(T)-\overline{x}(T))+\\
\int_{0}^{T} \sum_{i=1}^{l} \theta_i f^0_i(t,x(t),u(t))dt.
\end{array}
\]
By using (TC), we have 
\begin{equation}
\left.
\begin{array}{l}
\sum_{i=1}^{l} \theta_iD_Hg^0_i(\overline{x}(T))\cdot(x(T)-\overline{x}(T))\\
=p(T)\cdot (x(T)-\overline{x}(T))-\sum_{\alpha=1}^{m} \lambda_\alpha D_Hg^\alpha(\overline{x}(T)) \cdot (x(T)-\overline{x}(T))\\
-\sum_{\beta=1}^{q} \mu_\beta D_Hh^\beta(\overline{x}(T))\cdot(x(T)-\overline{x}(T))
\end{array}
\right\}
\end{equation}
Furthermore, by using (Si) and (S${\ell}$), we have for each $\alpha \in\{1,...,m\}$, $\lambda_\alpha g^\alpha(\overline{x}(T)) \le \lambda_\alpha g^\alpha(x(T))$.\\
Consequently, by using (S{\sc T}2), we have for all $\alpha \in\{1,...,m\}$, $\lambda_\alpha D_Hg^\alpha(\overline{x}(T)) \cdot(x(T)-\overline{x}(T))\ge 0$.\\
Besides, thanks to (S{\sc T}3), we have for all $\beta\in\{1,...,q\}$, $\mu_\beta D_Hh^\beta(\overline{x}(T))\cdot(x(T)-\overline{x}(T))=0$.\\
Hence, by using
$$
\begin{array}{l}
\sum_{i=1}^{l}\theta_i D_Hg^0_i(\overline{x}(T))\cdot(x(T)-\overline{x}(T))\\
\le p(T)\cdot (x(T)-\overline{x}(T)) \\
=\int_{0}^{T} \underline{d}(p(t)\cdot(x(t)-\overline{x}(t)))dt\\
=\int_{0}^{T} \underline{d}p(t)\cdot (x(t)-\overline{x}(t)) dt + \int_{0}^{T} p(t)\cdot \underline{d}(x(t)-\overline{x}(t)) dt\\
\le \int_{0}^{T} (H_B(t,\overline{x}(t),\overline{u}(t),p(t),(\theta_i)_{1\le i\le l})-H_B(t,x(t),u(t),p(t),(\theta_i)_{1\le i\le l}))dt+\\
\int_{0}^{T} (p(t)\cdot f(t,x(t),u(t))-p(t)\cdot f(t,\overline{x}(t),\overline{u}(t)))dt\\
=\int_{0}^{T} \sum_{i=1}^{l} \theta_i f^0_i(t,\overline{x}(t),\overline{u}(t))dt-\int_{0}^{T} \sum_{i=1}^{l} \theta_i f^0_i(t,x(t),u(t))dt. 
\end{array}
$$
Therefore, we have 
$$
\begin{array}{l}
\sum_{i=1}^{l}\theta_i J_i(x,u) \le \sum_{i=1}^{l} \theta_ig_i^0(\overline{x}(T))+\int_{0}^{T} \sum_{i=1}^{l} \theta_i f^0_i(t,\overline{x}(t),\overline{u}(t))dt\\
-\int_{0}^{T} \sum_{i=1}^{l} \theta_i f^0_i(t,x(t),u(t))dt+\int_{0}^{T} \sum_{i=1}^{l} \theta_i f^0_i(t,x(t),u(t))dt \\
=\sum_{i=1}^{l}\theta_i J_i(\overline{x},\overline{u}).
\end{array}.
$$
Consequently, $(\overline{x},\overline{u})$ is a solution of the following single optimization problem :
\[
(\mathcal{P}_\theta)
\left\{
\begin{array}{cl}
{\rm Maximize} & \sum_{i=1}^{l} \theta_i J_i(x,u) \\
{\rm subject \;  to} & (x,u)\in Adm({\mathcal B}).
\end{array}
\right.
\]
Now, we assume that $(\theta_i)_{1\le i\le l} \neq 0$. \\
We want to prove that $(\overline{x},\overline{u})$ is a weak Pareto optimal solution. We proceed by contradiction, we assume that $(\overline{x},\overline{u})$ is not a weak Pareto optimal solution i.e. there exists $(x,u)\in Adm({\mathcal B})$ such that for all $i\in \{1,...,l\}$, $J_i(x,u) > J_i(\overline{x},\overline{u})$.\\
Consequently, we have $\sum_{i=1}^{l} \theta_i J_i(x,u) > \sum_{i=1}^{l} \theta_i J_i(\overline{x},\overline{u})$. This is a contradiction with $(\overline{x},\overline{u})$ is a solution of $({\mathcal P}_\theta)$. \\
Next, we assume that for all $i\in\{1,...,l\},$ $\theta_i \neq 0$.\\
We want to prove that $(\overline{x},\overline{u})$ is a Pareto optimal solution. We proceed by contradiction, we assume that $(\overline{x},\overline{u})$ is not a Pareto optimal solution i.e. there exists $(x,u)\in Adm({\mathcal B})$ such that for all $i\in \{1,...,l\}$, $J_i(x,u)\ge J_i(\overline{x},\overline{u})$ and there exists $i_0\in\{1,...,l\}$, $J_{i_0}(x,u)>J_{i_0}(\overline{x},\overline{u})$.
Hence, we obtain that $\sum_{i=1}^{l} \theta_i J_i(x,u) > \sum_{i=1}^{l} \theta_i J_i(\overline{x},\overline{u})$. This is a contradiction with $(\overline{x},\overline{u})$ is a solution of $({\mathcal P}_\theta)$. 
\subsection{Proof of the Theorem \ref{thsb2}}
Notice that (S{\sc hb}2) implies (S{\sc hb}1).\\
Indeed, let $(x,u)\in Adm({\mathcal B})$.\\
We set $\theta=(\theta_i)_{1\le i \le l}$.
For all $t\in [0,T]$, for all $\varepsilon >0$ small enough, we have $\overline{x}(t)+\varepsilon(\overline{x}(t)-x(t)) \in \Omega$, 
therefore by using (MP.B) 
$$
\begin{array}{l}
\frac{1}{\varepsilon} (H_B^*(t,\overline{x}(t)+\varepsilon(\overline{x}(t)-x(t)),p(t),\theta)-H_B^*(t,\overline{x}(t),p(t),\theta)) \\
\ge \frac{1}{\varepsilon} (H_B(t,\overline{x}(t)+\varepsilon(\overline{x}(t)-x(t)),\overline{u}(t),p(t),\theta)-H_B(t,\overline{x}(t),\overline{u}(t),p(t),\theta)).
\end{array}
$$ 
Hence, since (S{\sc hb}2), (S{\sc i}2) and (S{\sc v}2), when $\varepsilon \rightarrow 0$ we have $D_{G,2}H_B^*(t,\overline{x}(t),p(t),\theta)\cdot (\overline{x}(t)-x(t)) \ge D_{G,2}H_B(t,\overline{x}(t),\overline{u}(t),p(t),\theta)\cdot (\overline{x}(t)-x(t))$.
Hence, by using (AE.B), we have 
\begin{equation}\label{sb1-1}
-D_{G,2}H_B^*(t,\overline{x}(t),p(t),\theta)\cdot (x(t)-\overline{x}(t)) \ge \underline{d}p(t)\cdot (x(t)-\overline{x}(t)).
\end{equation} 
Besides, for all $\varepsilon >0$ small enough, we have $\overline{x}(t)+\varepsilon(x(t)-\overline{x}(t)) \in \Omega$, 
hence by using (MP.B) and (S{\sc hb}2), we have
$$
\begin{array}{l}
\frac{1}{\varepsilon} (H_B^*(t,\overline{x}(t)+\varepsilon(x(t)-\overline{x}(t)),p(t),\theta)-H_B^*(t,\overline{x}(t),p(t),\theta)) \\
\ge H_B^*(t,x(t),p(t),\theta)-H_B^*(t,\overline{x}(t),p(t),\theta)\\
\ge H_B(t,x(t),u(t),p(t),\theta)-H_B(t,\overline{x}(t),\overline{u}(t),p(t),\theta).
\end{array}
$$
Hence, we have
$$
\begin{array}{l}
H_B(t,\overline{x}(t),\overline{u}(t),p(t),\theta)-H_B(t,x(t),u(t),p(t),\theta)\\
\ge \frac{1}{\varepsilon} (H_B^*(t,\overline{x}(t),p(t),\theta)-H_B^*(t,\overline{x}(t)+\varepsilon(x(t)-\overline{x}(t)),p(t),\theta)).

\end{array}
$$
Consequently, when $\varepsilon \rightarrow 0$, from (\ref{sb1-1}), we have \\
$H_B(t,\overline{x}(t),\overline{u}(t),p(t),\theta)-H_B(t,x(t),u(t),p(t),\theta) \ge \underline{d}p(t)\cdot(x(t)-\overline{x}(t))$.\\
Hence, the assumptions of the Theorem \ref{thsb1} are verified and the conclusions follow.\\
  %%%%%
  \subsection{Proof of the Theorem \ref{thsb3}}
  Notice that (S{\sc hb}3) implies (S{\sc hb}1).\\
Indeed, let $(x,u)\in Adm({\mathcal B})$, let $t\in [0,T]$, since $[(\xi,\zeta)\mapsto H_B(t,\xi,\zeta,p(t),(\theta_i)_{1\le i\le l})]$ is G\^ateaux differentiable and concave at $(\overline{x}(t),\overline{u}(t))$, we have \\$H_B(t,x(t),u(t),p(t),(\theta_i)_{1\le i\le l})-H_B(t,\overline{x}(t),\overline{u}(t),p(t),(\theta_i)_{1\le i\le l}) \\ \le D_{G,(2,3)}H_B(t,\overline{x}(t),\overline{u}(t),p(t),(\theta_i)_{1\le i\le l})\cdot(x(t)-\overline{x}(t),u(t)-\overline{u}(t))$.
Therefore, by using (AE.B) and (MP.B), we have \\
$D_{G,(2,3)}H_B(t,\overline{x}(t),\overline{u}(t),p(t),(\theta_i)_{1\le i\le l})\cdot(x(t)-\overline{x}(t),u(t)-\overline{u}(t))=-\underline{d}p(t)\cdot(x(t)-\overline{x}(t))$.   
Hence, (S{\sc hb}1) is verified. Therefore,the assumptions of the Theorem \ref{thsb1} are verified and the conclusions follow.


\begin{thebibliography}{99}
% 
\bibitem{BB1} M. Bachir, and J. Blot, Infinite Dimensional Infinite-horizon Pontryagin Principles for Discrete-time Problems, Set-Valued Var. Anal. 23 2015, pp. 43-54.
%
\bibitem{BB2} 
M. Bachir, and J. Blot, Infinite Dimensional Multipliers and Pontryagin Principles for Discrete-time Problems, Pure and Applied Functional Analysis, special issue on ÊInfinite Horizon Optimal Control and Dynamic Games. Vol 2, n3, 2017.
%
\bibitem{BJ} S. Bellaassali, and A. Jourani, Necessary optimality conditions in multiobjective dynamic optimization,
SIAM J. Control Optim, 42, 2004, pp. 2043--2046.
%
\bibitem{BY} J. Blot, and  H.Yilmaz, A generalization of Michel's result on the Pontryagin Maximum Principle, Journal of Optimization Theory and Applications, volume 183, 2019, pp. 792-812.
%
% 
\bibitem{BY1} J. Blot, and H. Yilmaz, Pontryagin Principle and Envelope theorem, https://doi.org/10.48550/arXiv.2206.13313, 2022.
%
\bibitem{BonKaya} H. Bonnel, and Y. Kaya, Optimization Over the Efficient Set of Multi-objective Convex Optimal Control Problems, Journal of Optimization Theory and Applications, 2010, 147 (1).
%
\bibitem{DJ} E. J. Dockner, S. Jorgensen, N. V. Long, and  G. Sorger, Differential Games in Economics and Management Science. Cambridge, MA: Cambridge University Press, 2000.
%
\bibitem{E}  J. C. Engwerda,  Necessary and sufficient conditions for Pareto optimal solutions of cooperative differential games, SIAM J. Control and Optimiz., vol. 48, no. 6, 2010, pp. 3859-3881.
\bibitem{F} TM. Flett, Differential analysis, Cambridge: Cambridge University Press, 1980.
%4
\bibitem{G} S. Gramatovici, Optimality conditions in multiobjective control problems with generalized invexity, Annals of University of Craiova, Math. Comp. Sci. Ser. Volume 32, 2005, pp. 150-157
%
\bibitem{H1} N. Hayek, Infinite horizon multiobjective optimal control problem in the discrete time case, Optimization, 60, 2011, 509-529.
\bibitem{H2} N. Hayek, A generalization of mixed problems with an application to multiobjective optimal control, J. Optim. Theory Appl., 150, 2011, pp. 498-515.
\bibitem{H3} N. Hayek, Infinite-horizon multiobjective optimal control problems for bounded processes, Discrete and Continuous Dynamical Systems - Series S DCDS-special issue on control and optimization, 2017.
\bibitem{H4} N. Hayek, Infinite-dimensional Infinite-horizon Multiobjective Optimal Control in Discrete Time, Pure and Applied Functional Analysis, Special issue on Control, Optimization and Variational Analysis, vol 4, Number 1,  2019, pp. 45-57.
\bibitem{L} G. Leitmann, Cooperative and Noncooperative Many Player Differential Games, Berlin, Germany: Springer Verlag, 1974.
\bibitem{M} O.L. Mangasarian, Nonlinear programming, SIAM, New York 1994. 
\bibitem{NgoH} T. Nhan Ngo, and N. Hayek, Necessary Conditions of Pareto Optimality for Multiobjective Optimal Control Problems under Constraints, Optimization, vol. 66,  Issue 2, 2017, pp 149-177.
%10
\bibitem{RE} P. V. Reddy, and J. C. Engwerda, Necessary and sufficient conditions for Pareto optimality in infinite horizon cooperative differential games, Contributions to Game Theory and Management, vol. 3, Graduate School of Management, St. Petersburg University, St. Petersburg, Russia, 2009, pp. 322-342.
\bibitem{SS} A. Seierstad, and K. Sydsaeter, Sufficient Conditions in Optimal Control Theory, International Economic Review, Vol. 18, No. 2,  1977 pp. 367-391.  
\bibitem{Oliv} V.A. de Oliveira, and G. Nunes Silva, On sufficient optimality conditions for multiobjective control problems, Journal of Global Optimization volume 64, 2016, pp. 721-744.
\bibitem{S} H. L. Stalford, Criteria for Pareto optimality in cooperative differential games, J. Optimiz. Theory and Applic., vol. 9, no. 6, 1972, pp. 391-398.
\bibitem{Zh} J. Zhu, Hamiltonian necessary conditions for a multiobjective optimal control problem with endpoint constraints, SIAM J. Control Optim, 39, 2000, pp. 97-112.
\end{thebibliography}
\end{document}